\documentclass{amsart}
\usepackage{amsmath,amssymb,amsthm,amsfonts,amscd,mathrsfs,mathtools}
\usepackage{enumerate} 
\usepackage[T1]{fontenc}
\usepackage{hyperref}

\newcommand\N{\mathbb{N}}
\newcommand\R{\mathbb{R}}

\newcommand{\GC}{\mathrm{GC}}
\newcommand{\newd}{\mathrm{d}}
\newcommand{\dsup}{\mathrm{d}_\mathrm{sup}}
\newcommand{\TC}{\mathrm{TC}}
\newcommand{\lTC}{\ell \mathrm{TC}}
\newcommand{\enrsecat}{\mathrm{secat}_\mathrm{ENR}}

\theoremstyle{plain}
\newtheorem{theorem}{Theorem}[section]

\newtheorem{corollary}[theorem]{Corollary}
\newtheorem{proposition}[theorem]{Proposition}

\theoremstyle{definition}
\newtheorem{definition}[theorem]{Definition}
\newtheorem{example}[theorem]{Example}

\theoremstyle{remark}
\newtheorem{remark}[theorem]{Remark}

\title{Geodesic complexity of motion planning}
\author{David Recio-Mitter}
\address{Department of Mathematics, Lehigh University\\Bethlehem, PA 18015, USA}
\email{dar318@lehigh.edu}
\date{\today}

\begin{document}
\maketitle

\begin{abstract}
We introduce the geodesic complexity of a metric space, inspired by the topological complexity of a topological space. Both of them are numerical invariants, but, while the TC only depends on the homotopy type, the GC is an invariant under isometries. We show that in some cases they coincide but we also develop tools to distinguish the two in a range of examples. To this end, we study what we denote the total cut locus, which does not appear to have been explicitly considered in the literature.

To the knowledge of the author, the GC is a new invariant of a metric space. Furthermore, just like the TC, the GC has potential applications to the field of robotics.
\end{abstract}

\section{Introduction} \label{sec:intro}

Almost two decades ago Farber introduced the topological complexity of a space to study the motion planning problem from robotics using topological tools \cite{Far03}. In short, the topological complexity is the smallest number of continuous rules necessary to motion plan on a given space, where a motion planning rule is a function which associates to each pair of points a path between them. Fewer rules signify higher stability (continuity) with respect to the input (pairs of points), which is why we seek to minimize the number of such rules.

To give the formal definition the following is needed: The path space $PX$ is the space of all paths on $X$ with the compact-open topology. The \textit{free path fibration} is the evaluation map $PX \to X \times X$ which sends each path $\gamma$ to the pair $(\gamma(0),\gamma(1))$.

\begin{definition}[Farber '03]
The {\em topological complexity} $\TC(X)$ of a space $X$ is defined to be the smallest $k$ for which there exists a decomposition into pairwise disjoint locally compact sets $X \times X = \bigcup_{i=0}^k E_i$ such that there are local sections $s_i \colon E_i \to PX$ of the free path fibration. 
\end{definition}

In the last two decades, the topological complexity has been computed for a multitude of different spaces by several authors. Furthermore, other variants of topological complexity have emerged, most of which essentially impose some restrictions on the motion planners. For instance, monoidal topological complexity requires the motion from each point to itself to be the constant path. It is an open question of Iwase and Sakai whether $\TC(X)=\TC_M(X)$ \cite{IS}. The equality has been shown for large classes of spaces \cite{Dra}.

From the point of view of applications, the latter restriction is very sensible: if a robot is already at the position to which it needs to move, it is a waste of energy for the robot to move at all. More generally, it is preferable for the motion planner to assign to each pair of points a path of minimal length between them. This is what motivates the definition of geodesic complexity below. First we need some preliminary definitions.

\begin{definition}\label{def:length}
Let $(X,d)$ be a metric space. The length of a path $\gamma\colon [0,1] \to X$ is given by
\[\ell(\gamma) = \sup_{0 = t_0 \le t_1 \le \ldots \le t_N = 1} \sum_{i=1}^N \newd(\gamma(t_{i-1}),\gamma(t_i)),\]
where the supremum is taken over all finite partitions of the interval $[0,1]$.
\end{definition}

\begin{definition}
Let $(X, \newd)$ be a metric space. We say that a path $\gamma$ is a \textit{geodesic} if there exists a number $\lambda$ such that
\[ \newd(\gamma(t),\gamma(t')) = \lambda |t-t'|\]
for all $0 \le t < t' \le 1$.
In particular the length of the path agrees with the distance between the endpoints: $\ell(\gamma)=\newd(\gamma(0),\gamma(1)) = \lambda$. Note that this is the shortest length a path from $\gamma(0)$ to $\gamma(1)$ could possibly have.
\end{definition}

\begin{remark}\label{rem:constant-speed}
Geodesics are shortest paths with constant speed, meaning that they are parametrized proportional to arc length. If a path $\gamma$ has minimal length, in the sense that $\ell(\gamma)=\newd(\gamma(0),\gamma(1))$, then it can be reparametrized to be a geodesic as defined above. This is discussed in Section \ref{sec:alternative-definition}.
\end{remark}

\begin{remark}
In the case when $X$ is a Riemannian manifold, the definition of geodesic above is equivalent to the definition of (smooth) \textit{minimizing} geodesic in the usual sense of Riemannian geometry. This is true for any path satisfying the definition of geodesic above, without the need to assume that it is smooth or even piecewise smooth. This follows from the fact that Riemannian manifolds are locally uniquely geodesic metric spaces, under the usual metric. This means that every point in a Riemannian manifold has a small ball around it in which any geodesic needs to coincide with the unique smooth geodesic with the same endpoints, and thus be itself smooth.
\end{remark}

\begin{definition}
Let $(X,\newd)$ be a metric space and let $GX \subset PX$ be the subspace of the free path space of $X$ consisting of all geodesics on $X$. We call $GX$ the \textit{geodesic path space}. Restricting the free path fibration $PX \to X \times X$ to $GX$ yields a map $\pi\colon GX \to X \times X$.
\end{definition}

\begin{definition}
The {\em geodesic complexity} $\GC(X, \newd)$ of a metric space $(X, \newd)$ is defined to be the smallest $k$ for which there exists a decomposition into $k+1$ pairwise disjoint locally compact sets $X \times X = \bigcup_{i=0}^k E_i$ such that there are local sections $s_i \colon E_i \to GX$ of $\pi$. We call such a collection of local sections a {\em geodesic motion planner} with $k+1$ domains of continuity.
\end{definition}

\begin{remark}
As was pointed out in Remark \ref{rem:constant-speed}, geodesics are shortest paths with constant speed. The constant speed condition is convenient when proving lower bounds because it makes the requirements for geodesic motion planners more rigid, but it is not essential in the definition of geodesic complexity. We show in \ref{thm:reparametrization} that dropping the ``constant speed'' requirement results in an equivalent definition. In short, a motion planner along shortest paths can always be modified to be a motion planner along shortest paths with constant speed, preserving continuity.
\end{remark}

The most commonly used definition of topological complexity requires an open cover with local sections over the open sets of the cover, instead of a decomposition into locally compact sets, as above. However, open sets do not work in the case of geodesic complexity; see Remark \ref{rem:opens}. The definition of topological complexity given above was also given by Farber and he showed that for reasonably nice spaces $X$ (for all Euclidean Neighborhood Retracts for instance) both definitions are equivalent \cite[Prop. 4.12]{Far08}. All spaces considered in this article are ENRs, in fact even manifolds.

\begin{remark}
By definition, $\TC(X) \le \GC(X, \newd)$. The topological complexity is a homotopy invariant (at least if $X$ is an ENR) and so it does not depend on the metric $\newd$. We will see that the geodesic complexity does depend on the metric and can in general differ dramatically from the topological complexity.

However, we also show that in many important examples, such as spheres, projective spaces, flat $n$-tori and the flat Klein bottle, the geodesic and the topological complexity agree; see Section \ref{sec:examples}.
\end{remark}

\begin{remark}
In most cases, the map $\pi$ is no longer a fibration. To see this, consider the example when $X$ is the circle (see Example \ref{ex:circle}). In that case, the fiber $\pi^{-1}((x,y))$ consists of either one geodesic (if $x\ne-y$) or two geodesics (if $x=-y$). However, if $\pi\colon GS^1 \to S^1\times S^1$ were a fibration, the fibers would have to be homotopy equivalent, since the base is connected.

Because $\pi$ is not a fibration, some of the techniques commonly used to compute the topological complexity do not work for geodesic complexity. However, in some cases $\pi$ is a \textit{level-wise stratified covering}. This concept is introduced in Section \ref{sec:cut-locus} in an attempt to characterize the properties of $\pi$ which allow us to find lower bounds for the geodesic complexity in the examples given in the sections \ref{sec:examples}, \ref{sec:embedded-torus} and \ref{sec:flat-spheres}. The ideas used in those sections are also used to find a general lower bound when $\pi$ is a level-wise stratified covering (satisfying some additional properties) in Corollary \ref{cor:general-lower-bound}, indicating that the characterization of level-wise stratified coverings is appropriate.
\end{remark}

There already exists a variant of topological complexity in the literature motivated by the idea of requiring the motion planners to be as efficient as possible. It was introduced by B\l{}aszczyk and Carrasquel-Vera in 2018 \cite{BCV} and it is called the {\em efficient topological complexity} $\lTC(M)$ of a compact orientable Riemannian manifold $M$.

%

However, geodesic complexity is in fact very different from efficient topological complexity, despite the similar heuristics: B\l{}aszczyk and Carrasquel-Vera show that $\TC(M) \le \lTC(M) \le \TC(M)+1$ for any closed orientable Riemannian manifold. In fact, they pose the (still open, to our knowledge) question of whether the first inequality is actually an equality $\TC(M) = \lTC(M)$ for all closed orientable Riemannian manifolds\footnote{They do give one example of a manifold with boundary with $\TC(M)=0$ and $\lTC(M)=1$, namely a closed hemisphere of the standard 2-sphere}. By contrast, we show that the difference between the geodesic complexity and the topological complexity is arbitrarily large, even for closed Riemannian manifolds.

\begin{theorem}\label{thm:gap}
For every $k\in\mathbb{N}$ there exists a closed Riemannian manifold $(M,g)$ such that $\GC(M)-\TC(M)\ge k$. In fact, $M$ can be chosen to be a sphere (with a non-standard metric).
Furthermore, for every $k\in\N$, there exists a metric $g_m'$ on $\R^{k+1}$ such that $\GC(\R^{k+1},g_m')\ge k$, while $\TC(\R^{k+1})=0$.
\end{theorem}

The reason geodesic complexity is very different from efficient topological complexity is that geodesic complexity depends mostly on the structure of the {\em total cut locus} of $X$, which is the subset $C \subset X \times X$ consisting of all pairs $(x, y)$
for which there is more than one geodesic $\gamma$ from $x$ to $y$; see Section \ref{sec:cut-locus}. This is because, by Theorem \ref{thm:cut-locus}, if a metric space is nice enough, there is a local section of $\pi$ on the complement the total cut locus. Thus the main challenge is to find a geodesic motion planner on the total cut locus itself.

However, efficient topological complexity does not require geodesic motion on the total cut locus: Because of the definition of efficient topological complexity, an efficient motion planner need not be geodesic on a set of measure zero, but the total cut locus of a closed Riemannian manifold is always a set of measure zero. Therefore, an efficient motion planner can be constructed by using any (not necessarily geodesic) motion planner on the total cut locus and the unique geodesic motion planner on the complement of the total cut locus (as a single motion planning set). B\l{}aszczyk and Carrasquel-Vera use this fact in \cite{BCV} and a proof can be found there.

In sections \ref{sec:examples}, \ref{sec:embedded-torus} and \ref{sec:flat-spheres} the geodesic complexity of several spaces are computed. It is worth noting that the lower bounds are proven by direct considerations of explicit motion planners (i.e. without recourse to algebra), which is very uncommon in the field of topological complexity and its variants. The ideas used there might be applicable to many further examples. We summarize the findings in the following theorems.

\begin{theorem}[Theorems \ref{thm:torus} and \ref{thm:klein}]
For the flat $n$-torus $T^n_{flat}$ and the flat Klein bottle $K$, the topological complexity and the geodesic complexity agree: $\GC(T^n_{flat})=\TC(T^n)=n$ and $\GC(K)=\TC(K)=4$.
\end{theorem}

It turns out that the torus has a different geodesic complexity if equipped with another very commonly used metric:

\begin{theorem}[Theorem \ref{thm:embedded-torus}]
Let $T_{emb}$ be the standard embedded torus in $\R^3$ and let $T$ be the flat 2-torus. Then
\[\GC(T_{emb}) = 3>\GC(T)=2.\]
\end{theorem}

\begin{theorem}[Theorem \ref{thm:flat-sphere}]
Let $W$ be the boundary of the 3-cube with the flat metric. This is a 2-sphere with a non-standard metric (a flat sphere) for which the geodesic complexity is different from the topological complexity:
\[\GC(W)\ge3>\TC(W)=\TC(S^2)=2.\]
\end{theorem}



The upper bounds on geodesic complexity in the previous theorems come from explicit geodesic motion planners. In particular, we give an optimal geodesic motion planner for the Klein bottle $K$. While it was known that a (not necessarily geodesic) motion planner with 5 sets exists by the general dimensional upper bound, no motion planner had been constructed explicitly. It is worth pointing out that we give a geometric proof of $\GC(K)=4$. In contrast, both proofs of the equality $\TC(K)=4$, the first by Cohen and Vandembroucq in 2018 \cite{CoVa} and the second by Iwase, Sakai and Tsutaya \cite{IST}, involve complicated algebraic calculations.

In Section \ref{sec:discussion} we formulate some open questions regarding geodesic complexity which arise naturally and seem to be promising problems for further research.


I would like to thank Diarmuid Crowley, Don Davis, Mark Grant, Mike Harrison and Jarek K\k{e}dra for very valuable discussions and suggestions. I would also like to thank the anonymous referees for their helpful suggestions.

\section{Are $\GC(X)$ and $\TC(X)$ equal?}\label{sec:gap}

A priori $\GC(X)$ and $\TC(X)$ could actually be the same for all reasonable metric spaces $X$ (assuming at least that $X$ is geodesically complete, to make sure that $\GC(X)$ is not automatically infinite). However, in this section we construct closed Riemannian manifolds $M$ for which the numbers $\GC(M)$ and $\TC(M)$ are arbitrarily far apart from each other. In fact, the constructed manifolds are spheres with non-standard Riemannian metrics. We also construct a Riemannian metric on $\R^n$ under which the geodesic complexity is unbounded as the dimension increases, see Theorem \ref{thm:gap}.

We need the following definition.

\begin{definition}
A subspace $K$ of a metric space $(X,\newd)$ is said to be \emph{convex} if for any pair of points $x,y\in K$, \textit{every} geodesic in $K$ between $x$ and $y$ lies entirely in $K$.
\end{definition}

\begin{theorem}\label{thm:convex}
If $K$ is a convex locally compact subspace of $(X,\newd)$, then $\TC(K) \le \GC(K) \le \GC(X)$.
\end{theorem}

\begin{proof}
A geodesic motion planner for $GX \to X\times X$ restricts to a geodesic motion planner on $GK \to K\times K$, because any geodesic with its endpoints in $K \times K$ has to lie entirely in $K$. Furthermore, for every locally compact $E_i$ in $X\times X$, the intersection $E_i \cap K\times K$ is also locally compact. This is because both the product and the intersection of locally compact Hausdorff spaces is locally compact.
\end{proof}

We will need the following theorem.

\begin{theorem}[Farber \cite{Far03}]\label{thm:sphere-torus}
Let $S^n$ be the $n$-sphere and $T^n$ the $n$-torus. Then
\[
\TC(S^n) = \begin{cases} 1 & \text{if} \;\; n \;\; \text{odd} \\ 2 & \text{if} \;\; n \;\; \text{even} \end{cases}
\]
and
\[\TC(T^n) = n.\]
\end{theorem}

The following example will motivate the proof of Theorem \ref{thm:gap}.

\begin{example}
Let $(S^3,g_m)$ be the result of glueing a hemisphere of the standard sphere $(S^3,g)$ onto each end of the standard cylinder $S^2\times [0,1]$. This space is sometimes known as a \textit{capsule}. After smoothing the edges, $(S^3,g_m)$ becomes a Riemannian manifold diffeomorphic to $S^3$. Because the submanifold $S^2\times \{\frac12\}\simeq S^2$ is convex, by Theorem \ref{thm:convex}: $\GC(S^3,g_m)\ge \TC(S^2)=2>1=\TC(S^3)$.
\end{example}

The proof of Theorem \ref{thm:gap} is essentially a generalization of the previous example.

\begin{proof}[Proof of Theorem \ref{thm:gap}]

Let $(S^{n+1},g)$ be the standard sphere of radius 1, with $n\ge2$.

Let $T^n\hookrightarrow S^{n+1}$ be an embedding of a torus with trivial normal bundle. Choose a tubular neighborhood $N_1$ of $T^n$ in $S^{n+1}$ and further tubular neighborhood $N_2$ around $N_1$. Because the normal bundle is trivial, $N_1$ is homeomorphic to a product $T^n\times (-1,1)$. Construct a new Riemannian metric $g_m$ on $S^{n+1}$ such that it coincides with $g$ outside of $N_2$ and it corresponds to the product metric on $N_1\cong T^n\times (-\text{diam}(T^n),\text{diam}(T^n))$, where diam$(T^n)$ is the diameter of $T^n$, i.e.\ the length of the longest geodesic in $T^n$. It is well known that a metric can be constructed in $N_2 - N_1$ to make $(S^{n+1},g_m)$ into a Riemannian manifold, using bump functions.

Given two points $x$ and $y$ in $( T^n  , g_m|_{T^n} )$, all geodesics between $x$ and $y$ must stay in $N_1$ because a path leaving $N_1$ would have a length strictly greater than the diameter of $(T^n,g_m|_{T^n})$. Furthermore, because $g_m|_{N_1}$ is the product metric it is clear that the geodesics all have to lie on $T^n\times\{0\}\subset N_1$. This shows that $( T^n  , g_m|_{T^n} )$ is convex in $( S^{n+1}  , g_m )$.

By Theorem \ref{thm:convex} we have $\GC(S^{n+1},g_m) \ge \TC(T^n) = n$, while $\TC(S^{n+1})$ equals either 1 or 2. The difference $\GC(S^{n+1},g_m) - \TC(S^{n+1}) \ge n-2$ can be made arbitrarily large by increasing $n$.

We can repeat the same argument with an embedding $T^k\hookrightarrow \R^{k+1}$ instead. In this way, we can construct a metric $g_m'$ on $\R^{k+1}$ such that $\GC(\R^{k+1},g_m') \ge \TC(T^k) = k$.

Linear interpolation yields a continuous motion planner on $\R^{k+1}$, showing that $\TC(\R^{k+1})=0$. 
\end{proof}


%
%

\begin{remark}
The arguments in the previous proof could be used to find lower bounds on the geodesic complexity of Riemannian manifolds constructed using any embedding of a Riemannian manifold into another which has a trivial normal bundle (with arbitrary codimension).

However, these examples could be seen as increasingly artificial, because the construction is very ad hoc. In sections \ref{sec:embedded-torus} and \ref{sec:flat-spheres} we exhibit some more naturally arising examples for which the topological complexity and the geodesic complexity also differ. The technique used in those examples is different from the one used in this section and is introduced in Section \ref{sec:cut-locus}.
\end{remark}

\section{The total cut locus and lower bounds for GC}\label{sec:cut-locus}

Let $(X, \newd)$ be a metric space. The following definition is very useful when studying $\GC(X)$.

\begin{definition}
The {\em total cut locus} of $X$ is the subset $C \subset X \times X$ consisting of all pairs $(x, y)$
for which there is more than one geodesic $\gamma$ from $x$ to $y$. The {\em cut locus of a point} $x\in X$ is the subset $C_x\subset X$ consisting of all $y$ in $X$ such that $(x,y)$ is in $C$.
\end{definition}


The cut locus of a point is treated in many differential geometry textbooks, such as \cite{Lee}. However, the author was not able to find any treatment of the total cut locus (by any name) in the literature. It should be mentioned that when $X$ is a Riemannian manifold the cut locus of a point as defined above does not contain first conjugate points along a unique minimizing geodesic, in contrast to the definition in differential geometry (see \cite{Lee}). The definition above can be found in \cite{BH} for instance.

The geodesic complexity seems to depend entirely on the nature of the total cut locus. Concretely, the following theorem shows that for nice metric spaces there exists a local section of $\pi\colon GX \to X\times X$ on the complement of the total cut locus. This means that finding a geodesic motion planner on the total cut locus is the hardest part.

\begin{definition}
We say that a metric space $(X,\newd)$ is a \textit{geodesic space} if every pair of points $x, y \in X$ is connected by a geodesic. This is clearly a necessary condition for the geodesic complexity of a metric space to be finite.
\end{definition}

\begin{theorem}\label{thm:cut-locus}
Let $(X,\newd)$ be a locally compact geodesic complete metric space. Then the map $\pi\colon GX \to X\times X$ has a local section over the complement of the total cut locus. In particular, if the total cut locus of $X$ is empty, then $\GC(X)=0$.
\end{theorem}

\begin{proof}

Recall that a metric space $(X,\newd)$ is said to be \textit{proper} if all closed balls are compact. We say that a metric space $(X,\newd)$ is a \textit{length space} if the distance between every pair of points $x, y \in X$ is equal to the infimum of the length of the paths joining them.

By \cite[Chapter I, 3.8 Corollary]{BH} a length space is proper if and only if it is complete and locally compact. We assume that $X$ is a geodesic space, which means that the distance between two points is realized by the length of a path between them, thus $X$ is a length space. Because $X$ is also assumed to be complete, it is a proper metric space.

Furthermore, \cite[Chapter I, 3.12 Lemma]{BH} states that for a proper metric space the unique geodesics connecting the pairs of points in $X\times X-C$ vary continuously. Thus, there is a local section of $\pi \colon GX \to X\times X$ over $X\times X-C$.

In particular, if $C=\emptyset$, this yields a motion planner with a single motion planning rule $E_0 = X\times X$ (it is well known that a product of locally compact sets is itself a locally compact). This implies that $\GC(X)=0$.
\end{proof}

\begin{remark}
Theorem \ref{thm:cut-locus} states that when the total cut locus of a metric space $X$ is empty, then $\GC(X)=0$. It is worth noting that the converse implication does not hold. For example, consider the case when $X=I^2$ is a square with the $\ell_1$-metric $\newd_1((x_1,x_2),(y_1,y_2)) = |x_1-y_1| + |x_2-y_2|$ (also known as the Manhattan metric). The total cut locus of $(I^2,\newd_1)$ consists of all pairs $((x_1,x_2),(y_1,y_2))$ with $x_1\ne y_1$ and $x_2\ne y_2$. However, $\GC(I^2,\newd_1)=0$ because there is a global section of $\pi$ which maps each pair of points to the geodesic along the straight segment between the points (i.e.\ the unique geodesic between the points under the $\ell_2$ metric).
\end{remark}

\begin{remark}
In \cite[Chapter I, 3.14 Exercise]{BH} Bridson and Haefliger exhibit a complete geodesic space which has an empty total cut locus and which is even (Lipschitz-1) contractible to a point, but for which the map $GX \to X \times X$ does not have a section. There is a unique geodesic over each pair of points, but these geodesics do not vary continuously.

Naturally (in light of the previous theorem), the space given in the example is not locally compact.
\end{remark}

\textbf{Question:} Is there a locally compact space $X$ on which there is a unique geodesic connecting each pair of points $x, y \in X$ but for which the unique geodesics do not vary continuously over $X \times X$ (implying $\GC(X)\ge1$)? Of course, by the previous theorem, such a space would not be complete.
\vspace{.2cm}

We study the total cut locus of several examples in sections \ref{sec:examples}, \ref{sec:embedded-torus} and \ref{sec:flat-spheres} to find lower bounds for the geodesic complexity. These examples seem to point to a general method to find lower bounds for the geodesic complexity of a metric space in cases in which the preimages $\pi^{-1}((x,y))$ are finite for all $(x,y) \in X \times X$, using completely different ideas from the ones used in Section \ref{sec:gap}. We suggest the definition of a level-wise stratified covering below and the proof of Theorem \ref{thm:general-lower-bound} as a framework in which to understand those examples. More work needs to be done in that direction.

As was pointed out in the introduction, the map $\pi \colon GX \to X \times X$ is not a fibration. However, if the preimages $\pi^{-1}((x,y))$ are finite for all $(x,y) \in X \times X$ then the map $\pi$ sometimes is a level-wise stratified covering in the following sense.

\begin{definition}
 A \textit{stratification} of $B$ is a decomposition
\[ B= \bigcup_{i=1}^N S_i \]
into disjoint subsets such that $\overline{S_i} = \bigcup_{j \ge i} S_j$, for $i = 1,\ldots,N$ and such that the following \textit{frontier condition} is satisfied: Let $S$ and $S'$ be path components of each set $S_i$, which we call the {\em strata}. Then $S' \cap \overline{S}\ne \emptyset$ implies that $S' \subset \overline{S}$.

We say that a Hausdorff space $B$ is a \textit{stratified space} if it admits a stratification.
\end{definition}

\begin{remark}
The definition of a stratified space above is slightly stronger than the definition given in \cite[Definition 2.2.7]{Fri}. Friedman's definition would be equivalent to the definition above if we only required  that each $\bigcup_{j \ge i} S_j$ be closed and that the frontier condition be satisfied, but do not require that $\overline{S_i} = \bigcup_{j \ge i} S_j$.
\end{remark}

\begin{definition}
We say that a surjective map $p\colon E \to B$ is a \textit{level-wise stratified covering} if the conditions I, II and III below are satisfied. It takes some work to state the conditions, but they can be roughly summarized as:
\begin{enumerate}
\item The space $B$ is a stratified space and restricting $p$ to each stratum yields a finite-sheeted covering.
\item When one stratum of $B$ lies in the closure of another stratum, the corresponding coverings fit together along their sheets in a nice way.
\end{enumerate}

\textbf{Condition I:} The space $B$ admits a stratification $(S_i)_i$ such that every $x$ in $B$ has a neighborhood $U$ which is a stratified space with a stratification $(U\cap S_i)_i$. Furthermore, the restriction $p|_{S_i}\colon p^{-1}(S_i) \to S_i$ is a covering and the neighborhood $U$ can be chosen such that the map $p|_{U\cap S_i}\colon p^{-1}(U\cap S_i) \to U\cap S_i$ is a trivial covering for each $i$ (even when $x$ is not in $S_i$).

Therefore, given a decomposition $U\cap S_i = \bigsqcup_{\alpha\in\Pi_i} P_{i,\alpha}$ into strata, the total space $p^{-1}(P_{i,\alpha})$ of each covering $p|_{P_{i,\alpha}}$ is a disjoint union $\bigsqcup_{\sigma\in\Sigma_{i,\alpha}} P^\sigma_{i,\alpha}$ of copies of the base space. Here $\Sigma_{i,\alpha}$ parametrizes the set of sheets of the covering $p|_{P_{i,\alpha}}$.

Furthermore, since $(U\cap S_i)_i$ is a stratification, we can associate to it a partially ordered set (or \textit{poset}) $(\mathcal{P}_U,\prec)$. The elements of $\mathcal{P}_U$ are the path-components $P_{i,\alpha}$ and the partial order $\prec$ is given by: $P_{i,\alpha} \prec P_{j,\beta}$ if and only if $\overline{P_{i,\alpha}} \supset P_{j,\beta}$. Anti-symmetry follows from the assumption that $U\cap S_i$ is a stratification.

Note that because of the frontier condition, the closure (within $U$) of each path-component can be written as
\[\overline{P_{i,\alpha}} = \bigcup_{P_{i,\alpha} \prec P_{j,\beta}} P_{j,\beta} .\]

\textbf{Condition II:} We assume that, whenever the closure (within $U$) of a stratum $P_{i,\alpha}$ contains  another stratum $P_{j,\beta}$, then there exist injective maps $\phi_{i,\alpha}^{j,\beta}\colon \Sigma_{i,\alpha} \to \Sigma_{j,\beta}$ specifying how the trivial coverings over the strata $p|_{P_{i,\alpha}}$ and $p|_{P_{j,\beta}}$ fit together: Each sheet $\sigma\in\Sigma_{i,\alpha}$ accumulates at the sheet $\phi_{i,\alpha}^{j,\beta}(\sigma)$. Concretely, the following restriction of the map $p$ needs to be a homeomorphism.
\[\bigcup_{P_{i,\alpha} \prec P_{j,\beta}} P_{j,\beta}^{\phi_{i,\alpha}^{j,\beta}(\sigma)} \; \underset{\cong}{\overset{p|}{\longrightarrow}}\; \bigcup_{P_{i,\alpha} \prec P_{j,\beta}} P_{j,\beta} = \overline{P_{i,\alpha}}\]

Furthermore, we assume that the maps $\phi_{i,\alpha}^{j,\beta}$ satisfy $\phi_{j,\beta}^{k,\gamma} \circ \phi_{i,\alpha}^{j,\beta}= \phi_{i,\alpha}^{k,\gamma}$.

\textbf{Condition III:} The poset $(\mathcal{P}_U,\prec)$ can be decomposed into \textit{levels} such that the partial order $\prec$ is the transitive closure of those order relations going only between adjacent levels, which are called \textit{covering relations}. Each level $i$ consists of the path-components of $U \cap S_i \ne \emptyset$.

\end{definition}

\begin{remark}
Note that in the definition above the stratification $(U \cap S_i)_i$ of $U$ induces a stratification $(p^{-1}(U \cap S_i))_i$ of $p^{-1}(U)$. With these stratifications, the restriction of the level-wise stratified covering $p|_U$ satisfies the conditions in the definition of a stratified covering given in \cite{CuPa}, except that Curry and Patel consider \textit{conically} stratified spaces. However, the definition above imposes stronger restrictions on $p|_U$.
\end{remark}

Note that each relation $P_{i,\alpha} \prec P_{j,\beta}$ corresponds to precisely one map $\phi_{i,\alpha}^{j,\beta}\colon \Sigma_{i,\alpha} \to \Sigma_{j,\beta}$, where we set $\phi_{i,\alpha}^{i,\alpha} = id_{\Sigma_{i,\alpha}}$. Recall that a covering relation is a relation $P_{i-1,\alpha'}\prec P_{i,\alpha}$ between elements that belong to adjacent levels.

\begin{definition}
We say that the poset $(\mathcal{P}_U,\prec)$ is \textit{inconsistent} at $P_{i,\alpha}$ if the intersection of the images of all maps into $P_{i,\alpha}$ corresponding to covering relations is empty
\[\bigcap_{\alpha'}\text{im}(\phi_{i-1,\alpha'}^{i,\alpha})=\emptyset\]
and if there exists at least one such map $\phi_{i-1,\alpha'}^{i,\alpha}$.

We also call the corresponding set of maps inconsistent.
\end{definition}

\begin{definition}
The {\em locally compact sectional category} of a map $p \colon E \to B$, $\enrsecat(p)$ is defined to be the smallest $k$ for which there exists a decomposition into $k+1$ locally compact sets $B = \bigcup_{i=0}^k E_i$ such that there are local sections $s_i \colon E_i \to E$ of $p$.
\end{definition}

\begin{remark}
By definition, the geodesic complexity of a metric space $(X,\newd)$ is the locally compact sectional category of the geodesic path evaluation map $\pi \colon GX \to X \times X$.
\end{remark}

The ideas of the proof in Theorem \ref{thm:general-lower-bound} are used in the sections \ref{sec:examples}, \ref{sec:embedded-torus} and \ref{sec:flat-spheres} to find lower bounds for the geodesic complexity. However, in some examples the map $\pi$ is not understood in sufficient detail to verify all properties of a level-wise stratified covering. Because of this and for the sake of concreteness, we prove the lower bounds in the sections \ref{sec:examples}, \ref{sec:embedded-torus} and \ref{sec:flat-spheres} separately, instead of referencing Corollary \ref{cor:general-lower-bound} below, even if the same general idea is used as in the proof of Theorem \ref{thm:general-lower-bound}. Most of the work in those sections consists of determining an appropriate stratification and understanding how the different coverings over the strata fit together.

Concretely, in the examples of Section \ref{sec:examples} (the flat $n$-tori and the flat Klein bottle) the map $\pi$ is a level-wise stratified covering, as becomes apparent in the proofs. On the other hand, for the embedded torus, which we study in Section \ref{sec:embedded-torus}, the map $\pi$ is not a level-wise stratified covering strictly speaking, because it does not satisfy the conditions II and III at the conjugate pairs. However, it satisfies those conditions for almost all points, which is sufficient for the lower bound. In fact, the argument for the lower bound only requires conditions I and II to be satisfied in some neighborhood of one single point.

Finally, for the flat 2-sphere in Section \ref{sec:flat-spheres} the map $\pi$ seems to be a level-wise stratified covering, but establishing this is a hard problem. Thankfully, in order to determine the lower bound for the geodesic complexity of the flat 2-sphere it will be sufficient to understand the restriction of $\pi$ to a small subset of pairs of points.

\begin{theorem}\label{thm:general-lower-bound}
Let $p\colon E \to B$ be a level-wise stratified covering. Assume that there exists an open set $U$ in $B$ such that the poset $(\mathcal{P}_U,\prec)$ is finite and fulfils the conditions I, II and III. Furthermore, assume that the poset $\mathcal{P}_U$ is inconsistent at every element, except for elements in the bottom level (where the set of incoming maps is empty). Let $N$ be the number of levels of $\mathcal{P}_U$. 

Then $\enrsecat(X)\ge N-1$.

Furthermore, if each covering $p|S_i$ is trivial, each intersection $U \cap S_i$ non-empty and each $S_i$ is locally compact, then $\enrsecat(X) = N-1$.
\end{theorem} 

\begin{proof}

$ $\newline

\textbf{Lower bound:}
\vspace{.1cm}

Let $B = \bigcup_{i=0}^k E_i$ be a disjoint decomposition such that there is a local section of $p$ over each $E_i$. Let $U$ be an open set satisfying the assumptions of the theorem. We need to show that $k\ge N-1$.


The proof proceeds by induction on the levels of the poset. The induction starts with the observation that every element of $P_{1,\beta}$ is in the closure of at least 1 $E_i$, because the sets $E_i$ cover $B$.

The induction step is the following: Assume that, for a given level $j$, every element of every path-component $P_{j,\beta}$ is in the closure of at least $j$ different sets $E_i$. Then every element of every path-component $P_{j+1,\beta'}$ in the next level is in the closure of at least $j+1$ different sets $E_i$.

The induction culminates in the statement that every element of $P_{N,\beta}$ is in the closure of $N$ different sets $E_i$. This implies that the decomposition above contains at least $N$ sets $E_i$ and thus $k\ge N-1$.

We will now prove the induction step. Assume, for the sake of contradiction, that there is an element $y$ of a path-component $P_{j+1,\beta'}$ which is not in the closure of $j+1$ different sets $E_i$; that is, there is a neighborhood $V$ which is contained in $\bigcup_{i=0}^{j-1} E_i$. We may assume that $y$ is in $E_0$.

By assumption, there exists a $P_{j,\beta}$ such that $P_{j,\beta} \prec P_{j+1,\beta'}$. Let $(a_n)$ be a sequence in $P_{j,\beta} \cap V$ converging to $y$.

By the induction assumption, every element of $P_{j,\beta} \cap V$ is in the closure of $E_i$ for each $i=0,\ldots,l-1$. In particular, every element of $P_{j,\beta} \cap V$ is in the closure of $E_0$. For each $n$, let $(b^n_{m})$ be a sequence in $E_0\cap V$ converging to $a_n$.

Setting $c_n = b^n_n$ yields a sequence $(c_n)$ in $E_0\cap V$ converging to $y$. We may assume that $(c_n)$ lies entirely in one path-component $P_{i,\beta}$, since there are only finitely many path-components.

Let $s_0\colon E_0 \to E$ be a local section. By continuity, the sequence of lifts $(s_0(c_n))$ needs to converge to $s_0(y)$ as $n$ goes to infinity.

It follows from Condition II in the definition of a level-wise stratified covering that a lift $(\tilde c_n)$ of $(c_n)$ converges to a lift $\tilde{y}$ of $y$ if and only if almost all $\tilde c_n$ lie in the same sheet $P_{i,\alpha}^\sigma$ and $\tilde{y}$ lies in the sheet $P_{j+1,\beta'}^{\phi_{i,\alpha}^{j+1,\beta'}(\sigma)}$.

It also follows from Condition II that \[\phi_{j,\beta}^{j+1,\beta'} \circ \phi_{i,\alpha}^{j,\beta}= \phi_{i,\alpha}^{j+1,\beta'}.\] Therefore, $\phi_{i,\alpha}^{j+1,\beta'}(\sigma)$ is contained in the image of $\phi_{j,\beta}^{j+1,\beta'}$, for all possible $\beta$.

We can carry out the previous argument with every $P_{j,\beta}$ such that $P_{j,\beta} \prec P_{j+1,\beta'}$. Therefore, the index of the sheet containing the lift $s_0(y)$ needs to be contained in the image of each $\phi_{j,\beta}^{j+1,\beta'}$ (for all possible $\beta$, and with $j+1$ and $\beta'$ fixed). In particular the images of those maps $\phi_{j,\beta}^{j+1,\beta'}$ need to have a non-empty intersection.

However, by the assumption of the theorem the set of those maps $\phi_{j,\beta}^{j+1,\beta'}$ is inconsistent. This means that the intersection of the images of those maps is empty. This yields a contradiction.

\vspace{.2cm}
\textbf{Upper bound:}
\vspace{.1cm}

If the coverings over each stratum $S_i$ are trivial, then there exists a local section of $p\colon E \to B$ over each $S_i$. Assuming that every intersection $U \cap S_i$ is non-empty, each $S_i$ corresponds to exactly one of the $N$ levels of the poset $\mathcal{P}_U$. This implies that there are $N$ many $S_i$ and thus $\GC(X) \le N-1$.
\end{proof}

The following corollary is immediate.

\begin{corollary}\label{cor:general-lower-bound}
Let $\pi\colon GX \to X\times X$ be the evaluation map on geodesic paths and assume it is a level-wise stratified covering satisfying the conditions of the previous theorem. Then $\GC(X)\ge N-1$. Furthermore, if each covering $p|S_i$ is trivial, each intersection $U \cap S_i$ non-empty and each $S_i$ is locally compact, then $\GC(X) = N-1$.
\end{corollary}

\begin{remark}
Notice that all that was needed in the proof of the lower bound is the stratification of the neighborhood $U$. In fact, we only need to be able to find a series of nested sequences such as the ones used in the induction argument.

Understanding the global picture is very useful to find upper bounds, but in some cases a local argument is more feasible and also sufficient for the lower bound; see Section \ref{sec:flat-spheres}.
\end{remark}

The following is the simplest example of a map $\pi \colon GX \to X \times X$ being a level-wise stratified covering.

\begin{example}\label{ex:circle}
Let $S^1$ be the unit circle with the standard Riemannian metric. The preimage of a pair $(x,y)$ under the map $\pi\colon GS^1 \to S^1 \times S^1$ is either the unique geodesic from $x$ to $y$ (when $x$ and $y$ are not antipodal) or two geodesics, one going clockwise and one counter-clockwise (when $x$ and $y$ are antipodal). Restricting the map $\pi$ to the total cut locus clockwise $C$ yields a trivial 2-sheeted covering with the sheet $r$ consisting of all the clockwise half-circle geodesics and $l$ consisting of all the counterclockwise half-circle geodesics. Restricting the map $\pi$ to $S^1\times S^1-C$ yields a 1-sheeted covering (i.e.\ a homeomorphism).

The map $\pi\colon GS^1 \to S^1 \times S^1$ is a level-wise stratified covering. The stratification of $S^1 \times S^1$ is given by $S_1 = S^1 \times S^1-C$ and $S_2 = C$.

The total cut locus $C$ can be visualized as a diagonal circle inside the torus $S^1 \times S^1$. As we noted, there are two sheets over each pair $(x,y)$ in $C$. We can approach such a pair from within the complement of $C$ from two sides: on one side the unique geodesic goes in the clockwise direction and on the other side the unique geodesic goes in the counter-clockwise direction. The map $\pi$ is the result of gluing a 1-sheeted covering over the complement of $C$ with a 2-sheeted covering over $C$: The single sheet over $S^1\times S^1-C$ converges to a different sheet over $C$ when approaching from each side.

Let $(x,y)$ be in $C$ and let $U$ be a small open ball around $(x,y)$. Then $U \cap S_1 = P_{1,r} \sqcup P_{1,l}$, where $P_{1,r}$ consists of pairs for which the unique geodesic goes in the clockwise direction and $P_{1,l}$ consists of pairs for which the unique geodesic goes in the counter-clockwise direction. Those two strata are separated by the stratum $U \cap S_2 = P_2$.

The preimages $p^{-1}(P_{1,r}) = P_{1,r}^{\sigma_r}$ and $p^{-1}(P_{1,l}) = P_{1,l}^{\sigma_l}$ consist of a single sheet. These approach the two sheets $p^{-1}(P_2) \cong P_2^r \sqcup P_2^l$ according to $\phi_{1,r}^2(\sigma_r) = r$ and $\phi_{1,l}^2(\sigma_l) = l$.

Because the set of maps $\{ \phi_{1,r}^2 , \phi_{1,l}^2 \}$ is inconsistent, the previous theorem implies that $\GC(S^1) \ge 1$. This lower bound also follows from the well known fact $\TC(S^1)=1$ (see Theorem \ref{thm:sphere-torus}). However, this serves as a very simple example of a more general method to get lower bounds for the geodesic complexity of a space.

For the sake of concreteness, we will give (a very simple special case of) the proof of the lower bound in this example. Recall that a local section $s_i\colon E_i \to GS^1$ consists of a continuous choice of geodesic over each pair in $E_i$. No pair $(x,y)$ in $C$ can lie in the interior of a set $E_i$: Otherwise we could construct two sequences $(a_n)$ and $(b_n)$, in $P_{1,r} \cap E_i$ and $P_{1,l} \cap E_i$ respectively, converging to $(x,y)$. By continuity, $s_i(a_n)$ and $s_i(b_n)$ need to converge to $s_i((x,y))$. However, this is impossible because the limit of $s_i(a_n)$ lies in the sheet $r$ (is a clockwise path) and the limit of $s_i(b_n)$ lies in the sheet $s$ (is a counterclockwise path). This implies that $\GC(S^1)\ge1$.

Finally, $\GC(S^1)=1$ because there exists an obvious geodesic motion planner with two domains of continuity $E_0=S^1\times S^1 - C$ and $E_1=C$, which was already constructed by Farber to prove $\TC(S^1)=1$ \cite{Far03} (let $s_1((x,y))$ be the geodesic going clockwise, for instance).
\end{example}

\begin{remark}\label{rem:opens}
The previous example makes clear that defining geodesic complexity using an open cover (as opposed to a pairwise disjoint decomposition into locally compact sets) would not work, given that if a set contains a point of the total cut locus in its interior then it does not admit a local section of $\pi$.
\end{remark}

\section{Spaces for which TC=GC}\label{sec:examples}

Some of the first spaces for which Farber computed the topological complexity are the spheres, see Theorem \ref{thm:sphere-torus}. To prove the optimal upper bound for $\TC(S^n)$ he constructed an explicit motion planner. Because that motion planner is geodesic (for the standard Riemannian metric on the sphere), the upper bound extends to $\GC(S^n)$. Together with the lower bound $\TC(X)\le\GC(X)$ this yields:

\begin{proposition}
Let $S^n$ be the standard $n$-dimensional sphere. Then $\GC(S^n)=\TC(S^n)$.
\end{proposition}

Not long after the computation of $\TC(S^n)$, Farber, Tabachnikov and Yuzvinsky uncovered a surprising link between the topological complexity $\TC(\R P^n)$ of real projective spaces $\R P^n$ and their immersion dimension, given in the following theorem. Recall that the immersion dimension $\text{Immdim}(M)$ of a smooth manifold $M$ is the smallest $k$ such that $M$ can be immersed into $\R^k$.

\begin{theorem}[Farber--Tabachnikov--Yuzvinsky \cite{FTY}]\label{thm:projective}
Let $\R P^n$ be the $n$-dimensional projective space. Then
\[
\TC(\R P^n) = \begin{cases} n & \text{if} \;\; n =1,3,7  \\ \rm{Immdim}(M) & \text{otherwise} \end{cases}.
\]
\end{theorem}

Farber, Tabachnikov and Yuzvinsky give a motion planner in \cite[Theorem 7.3]{FTY} which realizes the upper bound $\TC(\R P^n)\le\rm{Immdim}(\R P^n)$. While that motion planner is not geodesic, it can be easily modified to become geodesic. The same ideas can be used to give optimal geodesic motion planners in the cases where $n$ equals 1, 3 or 7. Just as in the case of spheres we get:

\begin{proposition}
Let $\R P^n$ be the standard $n$-dimensional real projective space. The $\GC(\R P^n)=\TC(\R P^n)$.
\end{proposition}

The remainder of this section is devoted to computing the geodesic complexity of two Riemannian manifolds, the torus $T$ and the Klein bottle $K$, both equipped with the \textit{flat metric} (in Section \ref{sec:embedded-torus} we will see that the standard embedded torus has a different GC). We will show the lower bounds using the ideas mentioned in Section \ref{sec:cut-locus}, rather than using $\TC(X)\le\GC(X)$. In sections \ref{sec:embedded-torus} and \ref{sec:flat-spheres} we will see that the lower bound coming from this technique can be strictly better than the lower bound $\TC(X)\le\GC(X)$.


\begin{theorem}\label{thm:torus}
Let $T^n_{flat}$ be the $n$-torus equipped with the flat metric. Then \[\GC(T^n_{flat})=\TC(T^n)=n.\]
\end{theorem}

\begin{proof}

\textbf{Stratification and upper bound}


The $n$-torus $T^n_{flat}=(S^1)^n$ has a stratification $(S_k)_{1\le k \le n+1}$ with:

\[S_k=\{(x,y)\in T^n_{flat}\times T^n_{flat}\;|\; y_i=x_i+\pi \; \text{for precisely $k-1$ many $i$}\}\]

The sets $S_k$ were used in Cohen and Pruidze \cite{CP} as domains of continuity for an explicit motion planner in the torus, which is actually a geodesic motion planner. Note that the geodesics on $T^n_{flat}$ with the flat metric are the paths in which all coordinates move simultaneously at constant speed along the shortest arc in the corresponding $S^1$ factor. If $(x,y)$ is in $S_k$, there are precisely $k-1$ coordinates which are antipodal. For each of those $k-1$ coordinates we need to choose to move either clockwise or counterclockwise in that coordinate, which results in $2^{k-1}$ many geodesics between $x$ and $y$. If we vary $(x,y)$ within $S_k$ the geodesics vary continuously and if we approach $S_m$ from $S_k$ the geodesics over $S_k$ converge to geodesics over $S_m$. It is easy to see that $\pi\colon GT_{flat}\to T_{flat}\times T_{flat}$ is a level-wise stratified covering in the sense of Section \ref{sec:cut-locus}.

Cohen and Pruidze show in \cite[Proposition 3.3]{CP} that it is possible to make continuous choices of geodesic over each $S_k$, yielding a motion planner on $n+1$ sets. This implies the upper bound $\GC(T^n_{flat})\le n$.

\vspace{.2cm}
\textbf{Lower bound for $n$=2}
\vspace{.1cm}

The lower bound immediately follows from $\GC(T_{flat})\ge\TC(T)=2$. However, we will prove that $\GC(T_{flat})\ge2$ directly by studying the stratification given above.


We will first give a proof for $n=2$ which can be extended to all $n$ in a straightforward manner. Let $T_{flat}=T^2_{flat}$.

%
%
%

Assume, for the sake of contradiction, that we have a decomposition $T_{flat} \times T_{flat} = E_0 \cup E_1$ with local sections $s_0$ and $s_1$ of $\pi\colon GT_{flat}\to T_{flat}\times T_{flat}$ over $E_0$ and $E_1$ respectively.

Let $(x,y)$ be in $S_3$. We may assume that $(x,y)$ is in $E_0$. Let $U_\epsilon$ be a small ball around $y$. For the remainder of the proof, the point $x$ will be fixed. Let $\tilde E_i = \{ y'\in T_{flat} \;|\; (x,y') \in E_i \}$.

%


Note that because $(x,y)\in S_3$ we have $y_1=x_1+\pi$ and $y_2=x_2+\pi$. We may identify $T_{flat}$ with the quotient of a square centered around $x$, in which the corners represent $y$: $T_{flat}=[0,2\pi]^2 / \sim$ with $(0,x_2)\sim(2\pi,x_2)$ and $(x_1,0)\sim(x_1,2\pi)$. The four distinct geodesics from $x$ to $y$ (the sheets of the level-wise stratified covering $GT_{flat}\to T_{flat}\times T_{flat}$ over $(x,y)$) can then be characterized by the directions up-right, up-left, down-right and down-left (UR, UL, DR, DL); see Figure \ref{fig:chambers}. By abuse of notation, we will call geodesics which are very close to one of those four geodesics by the same name. For example, all geodesics very close to UR will also be denoted UR.

The cut locus $C_x$ of $x$ divides the ball $U_\epsilon$ into four open chambers with their boundaries intersecting at $y$, see Figure \ref{fig:chambers}. The chamber decomposition has a cell structure compatible with the stratification: $y$ is the vertex and $(x,y)$ is in $S_3$, if $y'$ is in an edge then $(x,y')$ is in $S_2$ and if $y''$ is in the interior of a chamber then $(x,y'')$ is in $S_1$.

For points $y''$ in the interior of a chamber of $U_\epsilon$, the unique geodesic between $x$ and $y''$ is one of UR, UL, DR, DL, depending on the chamber containing $y''$. In this way we can identify the chambers with the geodesics and label them UR, UL, DR, DL just as in Figure \ref{fig:chambers}.

\begin{figure}[htb!]
\begin{minipage}{.48\linewidth}
  \begin{center}
    \includegraphics[scale=1]{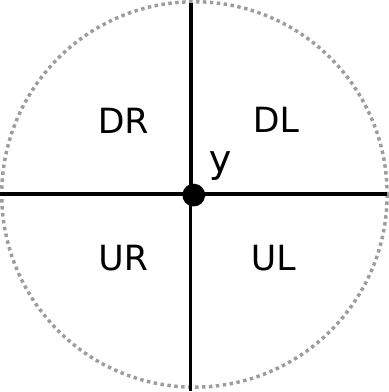}
  \end{center}
\end{minipage}%
\begin{minipage}{.48\linewidth}
  \begin{center}
    \includegraphics[scale=2]{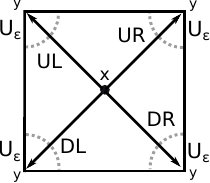}
  \end{center}
\end{minipage}
\caption{On the left is the small ball $U_\epsilon$ divided into chambers. On the right is the torus $T_{flat}^2$, seen as the square with opposite sides identified with each other. In particular, all the corner points represent the same point $y$ in $T_{flat}^2$.}
\label{fig:chambers}
\end{figure}


Essentially, the proof is rooted in the following observation: The pair $(x,y)\in S_3$ cannot be in the interior of $E_0$ because then $s_0$ could not be continuous at that point. To see why, suppose we choose $s_0((x,y))=\text{UR}$. Now note that there are pairs $(x,y'')$ arbitrarily close to $(x,y)$ such that a $\pi^{-1}((x,y''))=\text{UL}$. If $(x,y)$ were in the interior of $E_0$, then we could find such $(x,y'')$ in $E_0$ which are arbitrarily close to $(x,y)$. This would contradict continuity because $s_0((x,y''))$ has to be $\text{UL}$ and so it does not converge to $s_0((x,y))=\text{UR}$ as $y''$ approaches $y$.

Furthermore, to get a contradiction to continuity as in the previous paragraph we did not need to assume that $(x,y)$ is in the interior of $E_0$. To get the same contradiction, it suffices to assume that $E_0$ contains (for instance) pairs $(x,y_U)$ and pairs $(x,y_D)$ arbitrarily close to $(x,y)$ such that $\pi^{-1}((x,y_U))\in\{\text{UR},\text{UL}\}$ and that $\pi^{-1}((x,y_D))\in\{\text{DR},\text{DL}\}$. In order for $s_0$ to be continuous, we would have to choose $s_0((x,y))$ to be in the intersection $\{\text{UR},\text{UL}\}\cap\{\text{DR},\text{DL}\}=\emptyset$, which is impossible.


There is still one ingredient missing, because the contradiction pointed out in the previous paragraph does not show that we need more than two domains of continuity for a geodesic motion planner: If $E_0$ contained $(x,y)$ as an isolated point and $E_1$ contained the a neighborhood of $(x,y)$ (minus the point $(x,y)$ itself) the simple argument above would not yield a contradiction to continuity.

The key idea is that we can apply the contradiction argument recursively. The same argument as above can be used to show that no $(x,y')\in S_2$ can ever be in the interior of $E_0$ (or $E_1$). Because $T_{flat}\times T_{flat} = E_0\cup E_1$, this implies that all $(x,y')\in S_2$ are in the closure of both $E_0$ and $E_1$. By using the fact that $(x,y)$ can be approached by pairs $(x,y')\in S_2$ from all four edges, we can also approach $(x,y)$ by pairs contained in $E_0$ from different chambers (by a diagonal argument). As we pointed out above, that makes it impossible for $s_0((x,y))$ to be continuous.

Now we will follow the strategy outlined above to show the lower bound for $T_{flat}^2$ in greater detail. Afterwards we will extend it to all $T_{flat}^n$ using a proof by induction, where the induction step is essentially the same as in the argument for $T_{flat}^2$.

Let $y'$ be on the edge between chambers UR and DR. Assume, for the sake of contradiction, that $y'$ is in the interior of $\tilde E_0$. Then we would be able to find two sequences $(y_I''^k)$ and $(y_{II}''^k)$ in $\tilde E_0$ converging to $y'$ with $y_{I}''^k$ in UR and $y_{II}''^k$ in DR. By assumption, we have a local section $s_0\colon E_0\to GT_{flat}$.  By continuity $s_0((x,y_{I}''^k))=UR$ and $s_0((x,y_{II}''^k))=DR$ need to converge to the same path $s_0((x,y'))$, which yields a contradiction. Therefore, no $y'$ can lie in the interior of either $\tilde E_0$ or $\tilde E_1$. In other words, every $y'$ lying on an edge must be in the boundary of the closure of both $\tilde E_0$ and $\tilde E_1$.

Recall that $(x,y)$ is in $E_0$. We showed that every point on an edge between chambers lies in the closure of $\tilde E_0$ (and $\tilde E_1$). Using this and a diagonal argument we can construct a sequence $(y'^k)$ in $\tilde E_0$ converging to $y$, such that $(y'^k)$ is contained in a small neighborhood around the edge between the chambers UR and DR, for instance. Note that, by construction, $s_0((x,y'^k))$ must then converge to either the geodesic UR or the geodesic DR. By continuity, $s_0((x,y))$ must be either UR or DR.

However, if we assume instead that $(y'^k)$ is contained in a small neighborhood around the edge between the chambers UL and DL, the same argument would imply that $s_0((x,y))$ must be either UL or DL. This yields a contradiction, because $\{\text{UR},\text{DR}\}\cap\{\text{UL},\text{DL}\}=\emptyset$.

\vspace{.2cm}
\textbf{Lower bound for $n\ge3$}
\vspace{.1cm}

The lower bound for $n\ge3$ follows by iterating the argument for the case $n=2$.

Consider $T^3_{flat}=(S^1)^3$, for instance. Let $(x,y)\in S_4\subset T^3_{flat}\times T^3_{flat}$ and let $U_\epsilon$ be a small ball around $y$. This ball will have a chamber decomposition, just as for $n=2$, except that the square in Figure \ref{fig:chambers} has to be replaced by a cube. When $n=3$ there are eight 3-dimensional open chambers, separated by 2-dimensional walls, which intersect along edges, which in turn all intersect at the vertex at the center of the ball $U_\epsilon$.

Analogously to the case with $n=2$, we may show that every $y''$ in a 2-dimensional wall is in the closure of at least two $\tilde E_i$.

Now assume, for the sake of contradiction, that there exists a point $y'$ in an edge which has a neighborhood $V_\epsilon$ which is contained in $\tilde E_0 \cup \tilde E_1$. The point $y'$ can be approached by a sequence $(y''^k)$ contained in $V_\epsilon \cap W$, where $W$ is a wall adjacent to $y'$. Because each $y''^k$ is on a wall, it is in the closure of at least two $\tilde E_i$. Since each $y''^k$ is in $V_\epsilon\subset\tilde E_0 \cup \tilde E_1$, it must be in the closure of $\tilde E_0$ (and $\tilde E_1$). Using a diagonal argument we may construct a sequence $(\tilde y''^k)$ converging to $y'$ contained either in $W$ or a chamber adjacent to $W$, and such that $(\tilde y''^k)$ is contained in $\tilde E_0$. By continuity, $s_0((x,\tilde y''^k))$ must tend to $s_0((x,y'))$ as $k$ tends to infinity, meaning that $s_0((x,y'))$ is compatible with one of the two chambers adjacent to $W$. However, we could do the same argument with any wall adjacent to $y'$ and deduce that $s_0((x,y'))$ is compatible with a chamber which is adjacent that wall too. This yields a contradiction because there is no chamber adjacent to all four walls which are adjacent to the same edge. Therefore, every point $y'$ in an edge is in the closure of at least three $\tilde E_i$.

Next we assume, for the sake of contradiction, that the small ball $U_\epsilon$ around $y$ is contained in $\tilde E_0 \cup \tilde E_1 \cup \tilde E_2$. The point $y$ can be approached by a sequence $(y'^k)$ contained in $\tilde E_0 \cup \tilde E_1 \cup \tilde E_2$ and on an edge $L$. Because each $y'^k$ is in the closure of at least three $\tilde E_i$, it is in the closure of $\tilde E_0$ (and $\tilde E_1$ and $\tilde E_2$). Using a diagonal argument we may construct a sequence $(\tilde y'^k)$ converging to $y$ contained in $L$ or a chamber adjacent to $L$ and such that $(\tilde y'^k)$ is contained in $\tilde E_0$. By continuity, $s_0((x,\tilde y'^k))$ must tend to $s_0((x,y))$ as $k$ tends to infinity, meaning that $s_0((x,y))$ is compatible with one of the chambers adjacent to $L$. However, we could do the same argument with any edge and deduce that $s_0((x,y'))$ is compatible with a chamber adjacent to every edge. This yields a contradiction because there is no chamber adjacent to every edge. Therefore, the point $y$ is in the closure of at least four $\tilde E_i$. This implies that $\GC(T^3_{flat})\ge3$.

In the case of general $n$ the ball $U_\epsilon$ around $y$ is divided into $2^n$ many $n$-dimensional chambers separated by $(n-1)$-dimensional walls going through the center of the ball and meeting along $(n-2)$-dimensional walls and so on. The $k$-dimensional walls correspond to the $k$-skeleton of the $n$-dimensional cube intersected with small balls around the corner points. By induction on the argument above we can show that every point of a $k$-dimensional wall needs to be in the closure of $n-k+1$ many $\tilde E_i$, which implies that $\GC(T^n_{flat})\ge n$ (seeing the vertex as a 0-dimensional wall).
\end{proof}

In the following theorem we compute the geodesic complexity of the Klein bottle. While the lower bound in the theorem follows from $\GC(K)\ge\TC(K)$ and from $\TC(K)=4$, as shown by Cohen and Vandembroucq using rather complicated algebraic calculations, we show the lower bound directly without using $\TC(K)$.

The topological complexity of the higher-dimensional Klein bottles $K_n$ (introduced in \cite{Dav}) is unknown, but an argument similar to the one given in the proof of the theorem, albeit with much more complicated cut loci, yields the geodesic complexity of $K_n$ for all $n$; this is done by Davis and the author in \cite{DRM}.

\begin{theorem}\label{thm:klein}
Let $K_{flat}$ be the Klein bottle equipped with the flat metric. Then $\GC(K_{flat})=4$.
\end{theorem}

\begin{proof}
\textbf{Stratification}

The total cut locus of the Klein bottle is somewhat similar to the total cut locus of the torus but more complex. Before going through the following proof it is helpful to first read the proof of the previous theorem.

Let $K_{flat}= [0,1]^2 / \sim$ with $(0,x_2)\sim(1,1-x_2)$ and $(x_1,0)\sim(x_1,1)$.

The total cut locus of the torus is completely homogeneous, in the following sense. If we fix a point $x$ in the torus, then the cut  locus at $x$ is homeomorphic to $S^1 \vee S^1$. As we translate $x$ in the torus, the cut locus is simply translated along with $x$.

In the case of the Klein bottle, the cut locus of a point changes as we move the point around $K_{flat}$. To understand the total cut locus of $K_{flat}$ we will make use of the universal covering $p\colon \R^2 \to K_{flat}$. Given a point $x$ in $K_{flat}$, we can determine its cut locus $C_x$ in the following way.

Consider the set of all lifts $p^{-1}(x)$ in $\R^2$ under the universal covering. Now choose one point in $p^{-1}(x)$ and draw segments between that point and all the other lifts in $p^{-1}(x)$. Then the cut locus $C_x$ is the projection of the convex hull of the bisecting lines of all those segments. In Figure \ref{fig:torus} we see that in the case $x_2=0$ or $x_2=1/2$ the convex hull is a square which projects to $S^1\vee S^1 \subset K_{flat}$, the cut locus of $x$, just as in the case of the torus. However, when $x_2\ne0,1/2,1$ the convex hull is a (non-regular) hexagon, which projects down to a $\theta$ graph with edges of different lengths; see Figure \ref{fig:klein1}.

After the preliminaries above we are ready to describe how $C_x$ changes as we move $x$.

If we start moving a point $x$ with $x_2=0$ vertically, the cut locus $C_x$ continuously deforms from a wedge $S^1\vee S^1$ into a $\theta$ graph. The smallest edge gradually gets longer while one of the other edges of the graph gets shorter. Once $x$ reaches the other orientation-reversing ``meridian'' (for example we go up from $x_2=0$ to $x_2=1/2$), the new edge turns into a circle of the cut locus $S^1\vee S^1$, while the edge that was getting shorter has been contracted into the basepoint. 

On the other hand, if we move $x$ horizontally the cut locus is merely translated along.



\begin{figure}[htb!]
\begin{minipage}{.48\linewidth}
  \begin{center}
    \includegraphics[scale=0.8]{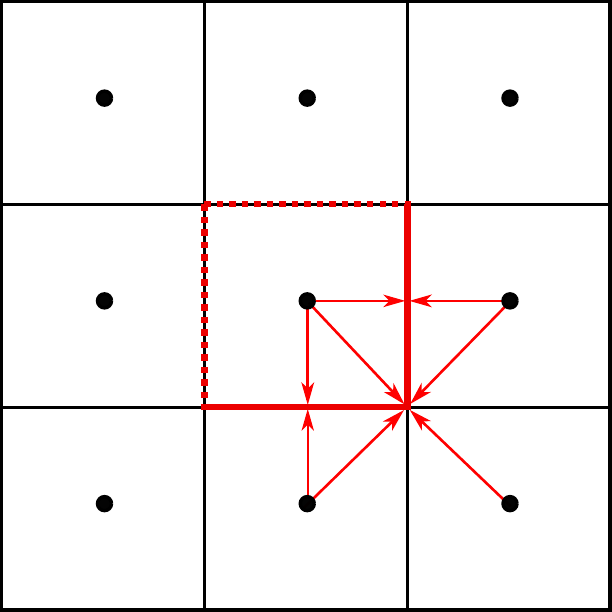}
  \end{center}
\end{minipage}%
\begin{minipage}{.48\linewidth}
  \begin{center}
    \includegraphics[scale=1 ]{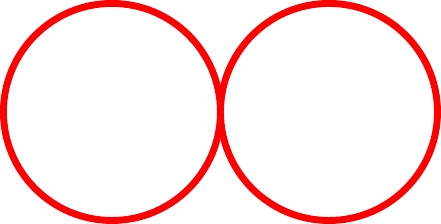}
  \end{center}
\end{minipage}
\caption{The cut locus of the point $(1/2,1/2)$ in the Klein bottle is a wedge of circles. The picture for the torus is exactly the same in this case.}
\label{fig:torus}
\end{figure}

\begin{figure}[htb!]
\begin{minipage}{.48\linewidth}
  \begin{center}
    \includegraphics[scale=0.8]{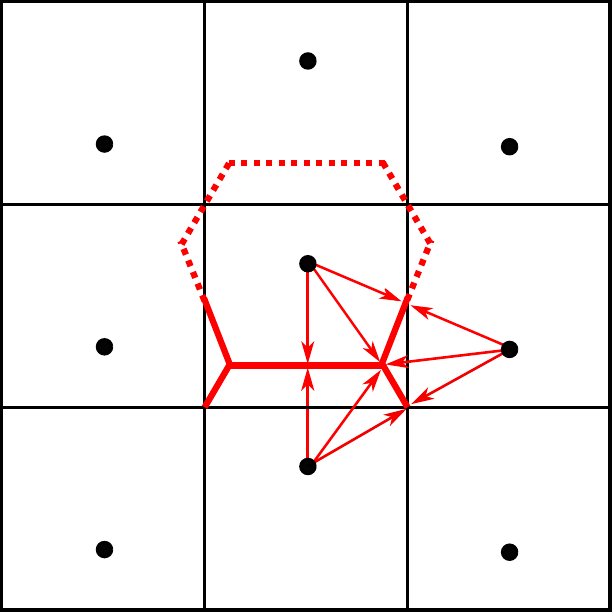}
  \end{center}
\end{minipage}%
\begin{minipage}{.48\linewidth}
  \begin{center}
    \includegraphics[scale=0.8 ]{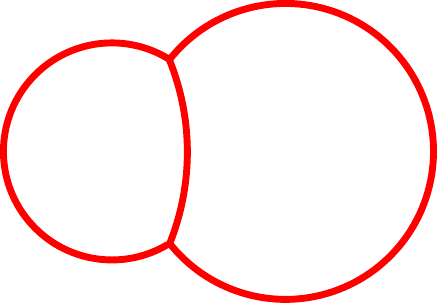}
  \end{center}
\end{minipage}
\caption{The cut locus of points between $(1/2,1/2)$ and $(1/2,1)$ in the Klein bottle. When the point moves up from $(1/2,1/2)$ a new edge appears at the vertex and then it keeps growing, while another edge gets shorter.}
\label{fig:klein1}
\end{figure}

\begin{figure}[htb!]
\begin{minipage}{.48\linewidth}
  \begin{center}
    \includegraphics[scale=0.8]{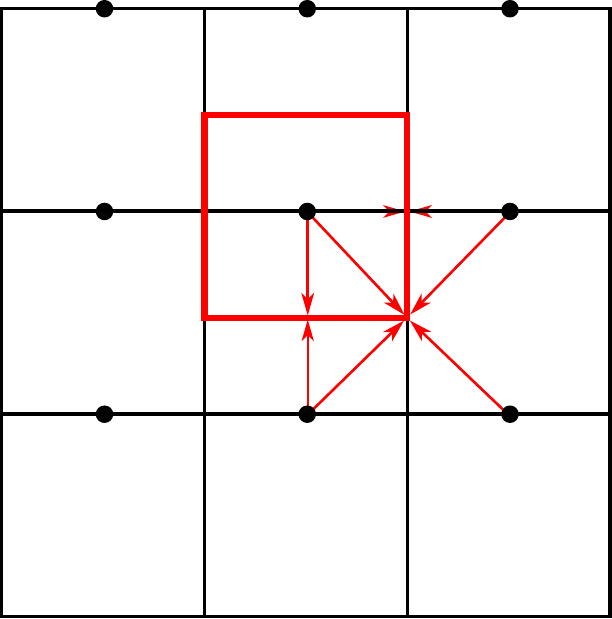}
  \end{center}
\end{minipage}%
\begin{minipage}{.48\linewidth}
  \begin{center}
    \includegraphics[scale=1 ]{wedge}
  \end{center}
\end{minipage}
\caption{The cut locus of the point $(0,1)$ in the Klein bottle is a wedge of circles again. The new edge that appeared when moving up from $(1/2,1/2)$ has completely replaced the old vertical edge.}
\label{fig:klein2}
\end{figure}

We construct the following stratification, which makes $\pi\colon GK_{flat}\to K_{flat}\times K_{flat}$ into a level-wise stratified covering.

\begin{enumerate}
\item The set $S_1$ is the complement of the total cut locus. That is, the pairs $(x,y)$ in $S_1$ are covered by a unique geodesics.
\item The set $S_2$ consists of the pairs $(x,y)$ covered by precisely two geodesics. This is the case precisely when $y$ lies in the interior of one of the edges of $C_x$, which is either $S^1\vee S^1$ or a $\theta$ graph.
\item The set $S_3$ consists of the pairs $(x,y)$ covered by precisely three geodesics. This is the case precisely when $x_2\ne0,1/2,1$ and $y$ lies on one of the vertices of the $\theta$ graph $C_x$.
\item The set $S_4$ consists of the pairs $(x,y)$ covered by precisely four geodesics. This is the case precisely when $x_2=0$, $x_2=1/2$ or $x_2=1$ and $y$ is the vertex of $S^1\vee S^1=C_x$.
\end{enumerate}

\vspace{.2cm}
\textbf{Upper bound}
\vspace{.1cm}

For the upper bound we will construct a geodesic motion planner with 5 sets. That is, a decomposition $K_{flat}\times K_{flat} = \bigcup_{i=0}^4E_i$ such that there is a local section of $\pi\colon GK_{flat}\to K_{flat}\times K_{flat}$ over each $E_i$.

Note that there is no local section of $GK_{flat}\to K_{flat}\times K_{flat}$ on $S_i$ for $i=2,3,4$ because we cannot make a consistent choice of a geodesic going ``up'' or ``down'' on the Klein bottle. To get sets $E_i$ for which we can define a geodesic motion planner we can ``cut'' the strata so as to make it impossible to go once around the Klein bottle along an orientation-reversing curve. We are going to divide the pairs $(x,y)$ into two sets, depending on whether $x$ lies in the annulus $A= \{ (x_1,x_2) \in K_{flat} \;|\; x_1\ne0 \}$.

We set

\[ S_i^A = \{ (x,y)\in S_i \;|\; x\in A \} \]

and $S_i^c= S_i - S_i^A$.

This yields another stratification of $K_{flat}\times K_{flat}$:

\begin{enumerate}
\item $\tilde E_0 = S_1$
\item $\tilde E_1 = S_2^A$
\item $\tilde E_2 = S_2^c \sqcup S_3^A$
\item $\tilde E_3 = S_3^c \sqcup S_4^A$
\item $\tilde E_4 = S_4^c$
\end{enumerate}

Intuitively, the sets range from more to less generic. The disjoint unions above indicate that the sets are topologically disjoint.



%

We will construct local sections $s_i\colon E_i \to GK_{flat}$ for all $i$ by constructing them separately on every path component of $
E_i$.

Conceptually, the remainder of the proof can be summarized as follows. The path components of the $E_i$ are either contractible or have the homotopy type of the circle. We will show that $GK_{flat} \to K_{flat}\times K_{flat}$ becomes a trivial covering when restricted to each path component, by constructing sections of all those coverings. In the case of contractible path components it is immediate that the restriction is a trivial covering. In the case of path components which are homotopy equivalent to the circle it boils down to showing that the monodromy action in the restriction is trivial (i.e.\ the loop going once around the circle lifts to a loop).

The section $s_0\colon E_0 \to GK_{flat}$ simply maps $(x,y)$ to the unique geodesic between $x$ and $y$.

The set $E_1 = S_2^A$ has three path components. The first path component contains those pairs $(x,y)$ such that $x$ is in the annulus $A$ and $y$ is being represented by a point which is in the interior of the horizontal edges either of the square or of the hexagon in figures \ref{fig:torus} and \ref{fig:klein1}. In this case, there are precisely two geodesics between $x$ and $y$, one going up and one going down, and $s_1((x,y))$ can be chosen to be the geodesic going up. Note that there is a consistent choice of up and down because $x$ has to remain in $A$. There are two further path components in $E_1$, containing those pairs $(x,y)$ such that $x$ is in the annulus $A$ and $y$ is being represented by a point which is in the interior of either the vertical edges of the square in Figure \ref{fig:torus} or in the slanted edges of the hexagon in Figure \ref{fig:klein1}. As the coordinate $y_2$ approaches $0$ (or $1$) or $1/2$ one of the slanted edges of the hexagon becomes a vertical edge of the square and the other one disappears. Both for vertical edges and slanted edges, there are precisely two geodesics between $x$ and $y$, one going to the right and one going to the left, and choosing $s_1((x,y))$ to be the geodesic going to the right yields a continuous map.

For the set $E_2 = S_2^c \sqcup S_3^A$ we are going to define the section $s_2 \colon E_2 \to GK_{flat}$ separately on $S_2^c$ and $S_3^A$. For $S_2^c$ the situation is analogous as for $S_2^A$ above, replacing $A$ by its complement (a circle). The set $S_3^A$ consists of pairs $(x,y)$ such that $x$ lies in $A$ and satisfies $x_2\ne0$ (and $x_2\ne1$) and $x_2\ne1/2$, and $y$ lies on one of the two vertices of the $\theta$-graph $C_x$. In terms of Figure \ref{fig:klein1}, $y$ is being represented by a corner of the hexagon. Note that $x$ is restricted to two disjoint open squares on which the corners of the hexagon remain completely separate. In fact, $S_3^A$ has four path components homeomorphic to $(0,1)^2$, because for each of the two open squares containing $x$, $y$ can be one of the two vertices in the $\theta$-graph $C_x$. For $(x,y)$ on a given path component, there will be precisely three geodesics from $x$ to $y$, either two going to the left and one to the right, or two going to the right and one to the left. Choosing $s_1((x,y))$ to be the geodesic going to the right in the former case and to the left in the latter case (for instance) yields a continuous map.

For the set $E_3 = S_3^c \sqcup S_4^A$ we are going to define the section $s_3 \colon E_3 \to GK_{flat}$ separately on $S_3^c$ and $S_4^A$. For $S_3^c$ the situation is analogous as for $S_3^A$ above, replacing $A$ by its complement (a circle). For $(x,y)$ in $S_4^A$ there are precisely four geodesics from $x$ to $y$. This situation corresponds to $y$ being represented by a point which is on a corner of the cube in Figure \ref{fig:torus}. The set $S_4^A$ has two path components homeomorphic to $(0,1)$, depending on whether $x_2=0$ or $x_2=1/2$. We choose the geodesic going up and to the right (represent $y$ by the upper right corner) for each path component.

For the set $E_4 = S_4^c$ the situation is analogous as for $S_4^A$ above, replacing $A$ by its complement (a circle).

\vspace{.2cm}
\textbf{Lower bound}
\vspace{.1cm}

For the lower bound we use the stratification $S_i$.

We want to show $\GC(K_{flat})\ge4$. Assume that we have a decomposition $K_{flat} \times K_{flat} = \bigcup_{i=0}^3E_i$ such that there exist local sections $s_i\colon E_i \to GK_{flat}$. Choose a point $(x,y)$ in $S_4$ and fix it for the remainder of the proof. We may assume that $(x,y)$ is in $E_0$.

The proof of the lower bound has two parts. First we show that every $(x,y)$ in $S_4$ is in the closure of every single $E_i$. We do this analogously to the proof of the lower bound for $\GC(T^n_{flat})$, by using a proof by contradiction argument recursively. In the second part we will use the fact that the $E_i$ are locally compact sets to show that this implies that every path component of $S_4$ must be contained in one of the sets $E_i$. This will yield a contradiction because no path component of $S_4$ admits a continuous geodesic motion planner.

Assume that there is a small ball $W_\epsilon$ around $(x,y)$ which only intersects $E_0$, $E_1$ and $E_2$ (so it does not intersect $E_3$). We will show that this contradicts the continuity of the local section $s_0$ at $(x,y)$ and thus every $(x,y)$ in $S_4$ has to be in the closure of every single $E_i$.

Choose a small ball $U_\epsilon$ around $y$ such that $\{x\}\times U_\epsilon$ is in $W_\epsilon$. The ball $U_\epsilon$ is divided into four chambers precisely as for the torus in the proof of Theorem \ref{thm:torus} (see Figure \ref{fig:chambers}).

Between $x$ and $y$ there are four geodesics. Making a local choice of orientation in the vertical direction (choice of up and down) the geodesics can be characterized by the directions up-right, up-left, down-right and down-left (UR, UL, DR, DL). By abuse of notation, we will denote geodesics which are very close to one of those four geodesics by the same name. For example geodesics which are very close to UR will also be denoted UR.

For points $y''$ in the interior of a chamber of $U_\epsilon$, the unique geodesic between $x$ and $y''$ is one of UR, UL, DR, DL, depending on the chamber containing $y''$. In this way we can identify the chambers with the geodesics and label them UR, UL, DR, DL just as in Figure \ref{fig:chambers}.

Now let $(x',y')$ be in $S_3\cap W_\epsilon$ and let $V_\epsilon$ be a small ball around $y'$ such that $\{x\}\times V_\epsilon$ is in $W_\epsilon$. Similarly to $U_\epsilon$, the ball $V_\epsilon$ is divided into chambers corresponding to the geodesics from $x'$ to $y'$, but into three chambers in this case; see Figure \ref{fig:chambers-klein}. Labelling the chambers by the geodesics associated to them as for $S_4$ above, we end up with four kinds of chamber decompositions of a small ball $V_\epsilon$ around a point $(x',y') \in S_3\cap W_\epsilon$:

\begin{enumerate}
\item $\{\text{UR,UL,DR}\}$
\item $\{\text{UR,DL,UL}\}$
\item $\{\text{DR,DL,UL}\}$
\item $\{\text{DR,UR,DL}\}$
\end{enumerate}

\begin{figure}[htb!]
\begin{minipage}{.48\linewidth}
  \begin{center}
    \includegraphics[scale=1]{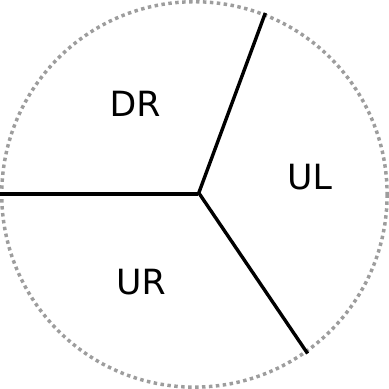}
  \end{center}
\end{minipage}%
\begin{minipage}{.48\linewidth}
  \begin{center}
    \includegraphics[scale=2]{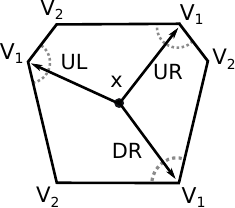}
  \end{center}
\end{minipage}
\caption{The small ball $V\epsilon$ divided into chambers.}
\label{fig:chambers-klein}
\end{figure}

\begin{figure}[htb!]
\begin{minipage}{.48\linewidth}
  \begin{center}
    \includegraphics[scale=1]{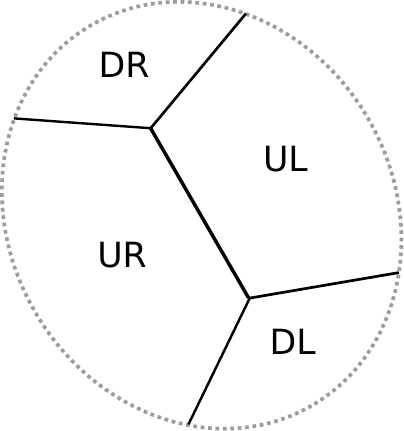}
  \end{center}
\end{minipage}%
\begin{minipage}{.48\linewidth}
  \begin{center}
    \includegraphics[scale=2]{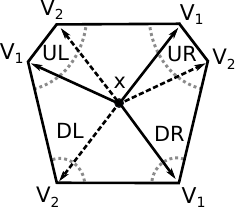}
  \end{center}
\end{minipage}
\caption{As $x'$ approaches $x$ from above the two 3-chamber decompositions around $S_3$ merge into one 4-chamber decomposition around $S_4$, as the hexagon turns into a square.}
\label{fig:chambers-klein2}
\end{figure}

\begin{figure}[htb!]
\begin{minipage}{.48\linewidth}
  \begin{center}
    \includegraphics[scale=1]{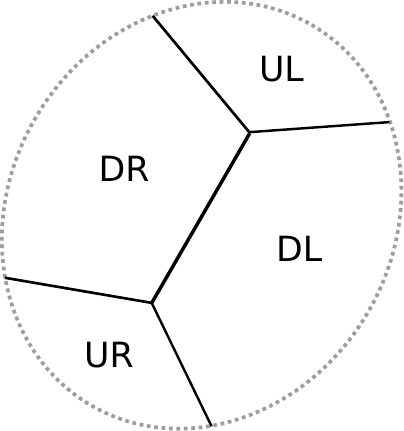}
  \end{center}
\end{minipage}%
\begin{minipage}{.48\linewidth}
  \begin{center}
    \includegraphics[scale=2]{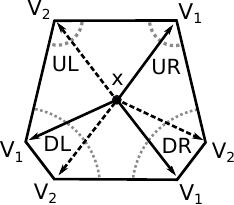}
  \end{center}
\end{minipage}
\caption{As $x'$ approaches $x$ from below the two 3-chamber decompositions around $S_3$ merge into one 4-chamber decomposition around $S_4$, as the hexagon turns into a square.}
\label{fig:chambers-klein3}
\end{figure}

Using the chambers described above and the exact same argument as in the proof of the lower bound for $\GC(T^2_{flat})$ we can show that every point of $S_3$ has to be in the closure of at least three different $E_i$. Otherwise, there would be a neighborhood of a pair $(x',y')$ in $S_3$ having a non-empty intersection only two $E_i$, which would in turn imply that there is no continuous local section at that point (just as in the argument for $T^2_{flat}$).

Because we assumed that $W_{\epsilon}$ only intersects the sets $E_0$, $E_1$ and $E_2$ and every point of $S_3$ has to be in the closure of at least three different $E_i$, all three of those sets need to accumulate at every point of $W_{\epsilon}\cap S_3$. In particular, $E_0$ accumulates at every point of $W_{\epsilon}\cap S_3$.

As was explained at the beginning of this proof and illustrated in the figures \ref{fig:torus}, \ref{fig:klein1} and \ref{fig:klein2}, the stratum $S_3$ accumulates at the stratum $S_4$. In the figures \ref{fig:chambers-klein2} and \ref{fig:chambers-klein3} we see how all four different kinds of 3-chamber decompositions for $S_3$ merge into the 4-chamber decomposition for $S_4$ which is looks like Figure \ref{fig:chambers}.

For instance, there is a sequence of points $((x'^k,y'^k))$ lying in $W_{\epsilon}\cap S_3$ and converging to $(x,y)$ such that $\pi^{-1}((x'^k,y'^k))=\{\text{UR,UL,DR}\}$ for all $k$. Using a diagonal argument we can construct another sequence $((x''^k,y''^k))$ in $E_0$ converging to $(x,y)$ and such that each $(x''^k,y''^k)$ is in a small neighborhood of $(x'^k,y'^k)$, which implies $\pi^{-1}((x''^k,y''^k))\in\{\text{UR,UL,DR}\}$ for all $k$. Recall that we abuse notation by denoting all geodesics which are close to each other by the same name, to simplify the notation. By continuity, $s_0((x''^k,y''^k))$ must converge to $s_0((x,y))$, which means that $s_0((x,y))\in\{\text{UR,UL,DR}\}$.

However, repeating the argument for the other types of chamber decomposition would imply that $s_0((x,y))\in\{\text{UR,UL,DR}\}\cap\{\text{UR,DL,UL}\}\cap\{\text{DR,DL,UL}\}\cap\{\text{DR,UR,DL}\} = \emptyset$. This yields a contradiction to our assumption that there is a small ball $W_\epsilon$ around $(x,y)$ which only intersects $E_0$, $E_1$ and $E_2$, because there is no choice of $s_0((x,y))$ that would make $s_0$ continuous at $(x,y)$. This shows that every point of $S_4$ lies in the closure of at least four different $E_i$.

The above argument is also illustrated slightly differently in Figure \ref{fig:lower-bound}.

\begin{figure}[hbt!]
\centering
\includegraphics[scale=1.5]{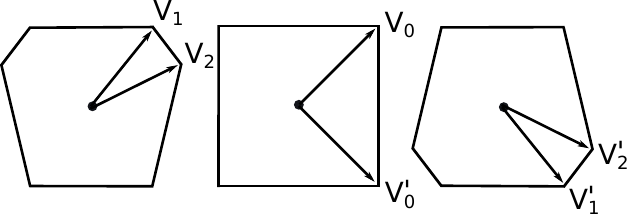}
\caption{As the vertical coordinate $x_2$ approaches $\frac12$, and $V_1$ ($V_1'$) and $V_2$ ($V_2'$) approach the vertex $V_0$ ($V_0'$), the only continuous choice of geodesics requires all geodesics go up (down) when $x_2$ is approaching $\frac12$ from above (below).}
\label{fig:lower-bound}
\end{figure}

In the final part of the proof we use the assumption that the $E_i$ are locally compact sets to improve the bound from $\GC(K_{flat})\ge3$ to $\GC(K_{flat})\ge4$. 

Assume, for the sake of contradiction, that we have a decomposition $X \times X = \bigcup_{i=0}^3E_i$ such that there exist local sections $s_i\colon E_i \to GX$. As we showed in part 1, every point of $S_4$ lies in the closure of at least four $E_i$. Because we are assuming that there are only four $E_i$ in the decomposition, $S_4$ must lie in the closure of every $E_i$.

Because $X$ is a metric space and thus Hausdorff, the $E_i$ have to be locally closed, which means that they are open in their closure. In particular the intersections $E_i\cap S_4$ yield a decomposition of $S_4$ into disjoint open sets. This implies that if $E_i$ intersects any point of $S_4$, it contains the entire path component of $S_4$ containing that point. However, no path component of $S_4$ admits a continuous (local) section because there is no consistent choice of up or down, due to the non-orientability of $K_{flat}$.

This yields a contradiction, implying that $\GC(K_{flat})\ge4$.
\end{proof}

\section{Embedded torus}\label{sec:embedded-torus}

In this section we prove that the geodesic complexity of the torus embedded in $\R^3$ in the standard way is higher than the geodesic complexity of the flat torus.

Specifically, the standard embedded torus $T_{emb}$ in $\R^3$ with meridian circle of radius $r=1$ and core circle of radius $R=2$ is given by

\[
T_{emb} = \left\{ (x,y,z)\in\R^3 \;\left|\; \left(\sqrt{x^2+y^2}-2\right)^2 + z^2 = 1 \right.\right\}
.\]

\begin{theorem}\label{thm:embedded-torus}
Let $T_{emb}$ be the embedded torus defined above and let $T_{flat}$ be the flat 2-torus. Then
\[\GC(T_{emb}) = 3>\GC(T_{flat})=2.\]
\end{theorem}

\begin{proof}
The equality $\GC(T_{flat})=2$ follows from Theorem \ref{thm:torus}.

To compute $\GC(T_{emb})$ we use the description of the cut locus of any point in $T_{emb}$ given by Gravesen, Markvorsen, Sinclair and Tanaka in \cite{GMST}:

The cut locus of a point $p = (x_0, 0, z_0)$ with $x_0 > 0$ on the torus is the union of
\begin{enumerate}[(i)]
\item the opposite meridian $y = 0, x < 0$,
\item a (piecewise $C^1$) Jordan curve which intersects the opposite meridian at a
single point and is freely homotopic to each parallel,
(see Figure 1 in Section 5 of \cite{GMST}) and, if $p$ is sufficiently far from the inner equator, i.e., if $x_0 > c_2$ for some positive constant $c_2$ ($> R - r = 2-1 = 1$),
\item a pair of subarcs of the parallel $z = -z_0$, each with a conjugate point of $p$ as one endpoint and joining
\begin{itemize} \item only the Jordan curve of (ii) if $c_2 < x_0 < c_1$ for some $c_1$, (see Figure 2 in Section 5 of \cite{GMST})
\item both of the above if $x_0 = c_1$ (see Figure 3 in Section 5 of \cite{GMST}) or
\item only the meridian of (i) if $c_1 < x_0$, (see Figure 4 in Section 5 of \cite{GMST}) at their other endpoint.
\end{itemize}
\end{enumerate}

In particular, the cut locus of any point in $T_{emb}$ is a graph. Let $v_0^p$ denote the vertex at the intersection between the opposite meridian and the parallel $z = -z_0$ in the cut locus of a point $p = (x_0, 0, z_0)$ as described above.

In Figure \ref{fig:embedded-torus} we illustrate the three different types of cut locus of a point $p$, in a neighborhood of $v_0^p$.

\begin{figure}[htb!]
\begin{minipage}{.3\linewidth}
  \begin{center}
    \includegraphics[scale=1.1]{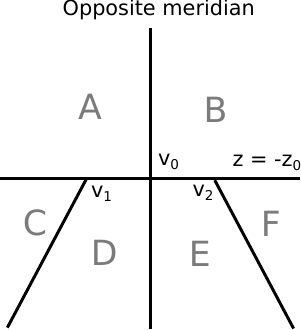}
  \end{center}
\end{minipage}%
\begin{minipage}{.3\linewidth}
  \begin{center}
    \includegraphics[scale=1.1]{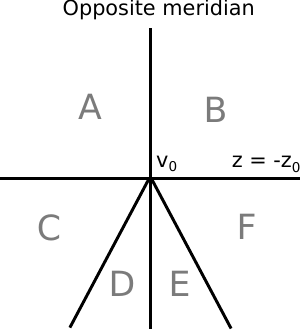}
  \end{center}
\end{minipage}%
\begin{minipage}{.3\linewidth}
  \begin{center}
    \includegraphics[scale=1.1]{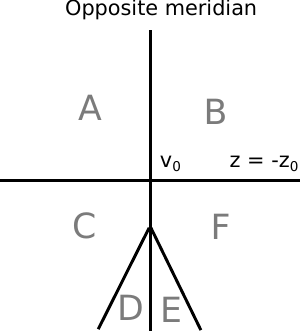}
  \end{center}
\end{minipage}
\caption{The three diagrams represent a neighborhood of the intersection point of the parallel $z=-z_0$ and the opposite meridian in the cut locus of a point $p=(x_0,0,z_0)$. The value of $x_0$ increases from the left to the right.}
\label{fig:embedded-torus}
\end{figure}
%
%

\vspace{.2cm}
\textbf{Lower bound}
\vspace{.1cm}

Assume, for the sake of contradiction, that we have a decomposition $T_{emb} \times T_{emb} = \bigcup_{i=0}^2E_i$ such that theres exist local sections $s_i\colon E_i \to GT_{emb}$.

Let $p$ be a point in $T_{emb}$ and $q$ be a vertex (of degree at least 3) in the cut locus of $p$. Analogously to the proofs of theorems \ref{thm:torus} and \ref{thm:klein} we can show that each such pair $(p,q)$ lies in the closure of $E_0$, $E_1$ and $E_2$. We may assume that $(p,v_0^p)$ is in $E_0$.

Let $v_1^p$ and $v_2^p$ be vertices of the cut locus of a point $p$ as in Figure \ref{fig:embedded-torus}. As the $x$ coordinate of $p$ increases, the vertices $v_1^p$ and $v_2^p$ merge together with $v_0^p$.

There are three geodesics between $p$ and $v_1^p$, each going through one of the three domains A, C and D. Similarly, there are three geodesics between $p$ and $v_2^p$, each going through one of the three domains B, E and F. As we vary $p$ and $v_1^p$ and $v_2^p$ converge to $v_0^p$, the geodesics $s_0((p,v_1^p))$ and $s_0((p,v_2^p))$ need to converge to the same geodesic $s_0((p,v_0^p))$. However, that is impossible, since the geodesic $s_0((p,v_0^p))$ would have to go through a domain contained in $\{\text{A,C,D}\}\cap\{\text{B,E,F}\}=\emptyset$.

\vspace{.2cm}
\textbf{Upper bound}
\vspace{.1cm}

There exists a geodesic motion planner on $T_{emb}$ with the following sets $E_i$.

\begin{itemize}
\item Let $E_0$ be the complement of the total cut locus of $T_{emb}$.
\item Let $E_1$ consist of those pairs $(p,q)$ such that $q$ lies in the interior of an edge of the cut locus $C_p$.
\item Let $E_2$ consist of those pairs $(p,q)$ such that $q$ lies on a vertex of the cut locus $C_p$ other than the vertex $v_0^p$. Recall that $v_0^p$ is the intersection point between the meridian and the Jordan curve.
\item Let  $E_3$ consist of the pairs $(p,v_0^p)$.
\end{itemize}

There are two geodesics between any pair in $E_1$, three geodesics between any pair in $E_2$ and either four or six geodesics between any pair in $E_3$.

Using the description of the cut locus given above we can make a continuous choice of geodesic over each of the sets $E_i$, just as for the Klein bottle in the proof of Theorem \ref{thm:klein}. See also the figures in \cite[Section 5]{GMST}.
\end{proof}

\section{Flat spheres}\label{sec:flat-spheres}

In Section \ref{sec:gap} we constructed a Riemannian metric on $n$-spheres with which the geodesic complexity is strictly greater than the topological complexity. The method used in that section does not work for the 2-sphere, however. It only works for $n\ge3$ because it relies on having an embedded $(n-1)$-dimensional submanifold $M$ with $\TC(M)>\TC(S^n)$. For $n=2$ we have that $\TC(S^2)=2$ and $\TC(M)\le2\text{dim}(M)=2$.

In this section we exhibit a metric on the 2-sphere which yields a higher geodesic complexity than that of the standard 2-sphere. This metric space was provided by Jarek K\k{e}dra as an example with a pathological total cut locus when the author was trying to get a better intuition about the cut locus of general metric spaces.

Note that the ideas used in the proof are completely different from the method of Section \ref{sec:gap}. Instead, we are using the same ideas as for the flat torus and the flat Klein bottle in Section \ref{sec:examples} and for the embedded torus in Section \ref{sec:embedded-torus}, which are conceptualized in Section \ref{sec:cut-locus}.

\begin{definition}
Let $W$ be the boundary of the 3-cube with the flat metric. The flat metric comes from regarding $W$ as a subset of the plane with the edges identified appropriately (see Figure \ref{fig:cube-coordinates}). This is a topological manifold which is homeomorphic to the 2-sphere. We call it the flat 2-sphere.
\end{definition}

The following theorem is somewhat surprising. Unlike for the spheres in Section \ref{sec:gap}, we did not construct the metric purposefully to get a higher geodesic complexity in this case.

\begin{theorem}\label{thm:flat-sphere}
Let $W$ be the flat 2-sphere. Then $\GC(W)\ge3>2=\TC(W)$.
\end{theorem}

\begin{remark}
It should be stressed that $W$ is a topological manifold and a metric space, but not a Riemannian manifold or even a smooth manifold. In fact, it is well-known that there cannot be a 2-sphere with a flat Riemannian manifold, as follows from the Gauss-Bonnet Theorem. That said, the theorem might still hold after slightly smoothing the edges and corners. The proof seems to carry over to that case, intuitively. However, explicitly describing the geodesics on such a space would require methods from differential geometry.
\end{remark}

\begin{proof}

The fact that $\TC(W)=\TC(S^2)=2$ follows from Theorem \ref{thm:sphere-torus} and the homotopy invariance of topological complexity (since $W$ and $S^2$ are both ENRs).

The total cut locus of $W$ is quite complicated. However, just as in the previous section, the argument relies on understanding the different geodesics between points arbitrarily close to a specific pair of points. Concretely, in this case it suffices to consider (a subset of) the neighborhood of a pair of opposite corners in $W \times W$.

\vspace{.2cm}
\textbf{Step 1: Geodesics between pairs approaching opposite corners}
\vspace{.1cm}

Let $(x,y)$ be a pair of points with $x$ and $y$ on opposite faces. Consider the coordinates $(x_1,x_2)$ and $(y_1,y_2)$ for opposite faces given in Figure \ref{fig:cube-coordinates}, in which the midpoints of each face acts as the origin and the squares have side length 1.

\begin{figure}[hbt!]
  \centering
    \includegraphics{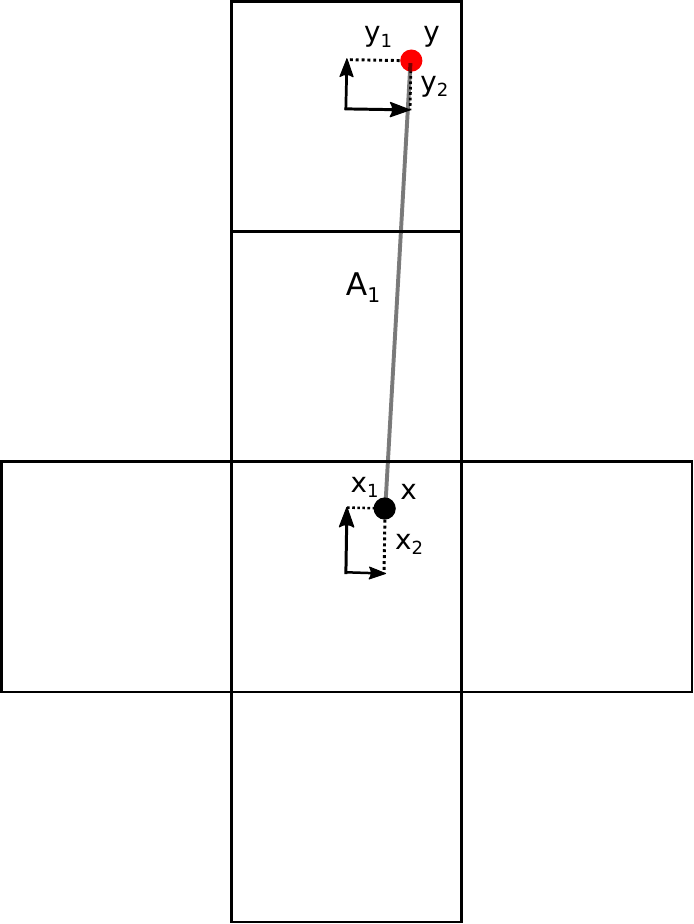}
    \caption{We introduce coordinates for points $x$ and $y$ on opposite faces.}
    \label{fig:cube-coordinates}
  \end{figure}  

As seen in Figure \ref{fig:cube-paths}, there are \textit{at most} 12 paths $A_1$, $A_2$, \dots , $A_{12}$ which could potentially be geodesics between $x$ and $y$.

\begin{figure}[hbt!]
  \centering
    \includegraphics{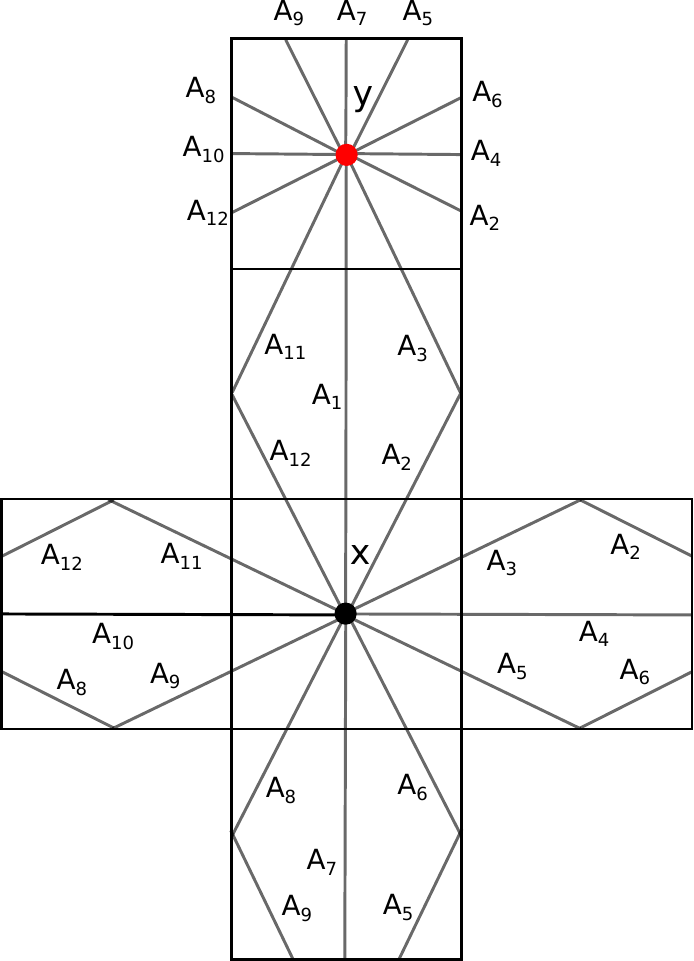}
    \caption{There are 12 paths $A_1$,\ldots, $A_{12}$ which are \textit{potentially} geodesics between two points on opposite faces, depending on the specific positions of the points. To make the picture more symmetric, $x$ and $y$ have been chosen to be on the midpoints of the faces. The points $x$ and $y$ can be anywhere on the two faces. However, note that some of the paths are not admissible for every choice of $x$ and $y$. That being said, the paths $A_1$, $A_4$, $A_7$ and $A_{10}$ are admissible for all $x$ and $y$ on opposite faces.}
        \label{fig:cube-paths}
  \end{figure}

Denoting the length of the path $A_i$ by $L_i$, we have the following.

\begin{itemize}
\item $L_1^2 = (x_1-y_1)^2 + (2-x_2+y_2)^2$
\item $L_2^2 = (1-x_1+y_2)^2 + (2-x_2-y_1)^2$
\item $L_3^2 = (1-x_2-y_1)^2 + (2-x_1+y_2)^2$
\item $L_4^2 = (x_2+y_2)^2 + (2-x_1-y_1)^2$
\item $L_5^2 = (1+x_2-y_1)^2 + (2-x_1-y_2)^2$
\item $L_6^2 = (1-x_1-y_2)^2 + (2+x_2-y_1)^2$
\item $L_7^2 = (x_1-y_1)^2 + (2+x_2-y_2)^2$
\item $L_8^2 = (1+x_1-y_2)^2 + (2+x_2+y_1)^2$
\item $L_9^2 = (1+x_2+y_1)^2 + (2+x_1-y_2)^2$
\item $L_{10}^2 = (x_2+y_2)^2 + (2+x_1+y_1)^2$
\item $L_{11}^2 = (1-x_2+y_1)^2 + (2+x_1+y_2)^2$
\item $L_{12}^2 = (1+x_1+y_2)^2 + (2-x_2+y_1)^2$
\end{itemize}

It can be readily seen that all the $L_i^2$ have a common summand $x_1^2+x_2^2+y_1^2+y_2^2+4$ once we multiply the squares out. Once we subtract that common summand, all expressions have a common factor of 2. To better compare the lengths $L_i$ we will consider the ``normalized'' square lengths $N_i = (L_i^2 - (x_1^2+x_2^2+y_1^2+y_2^2+4))/2$:

\begin{itemize}
\item $N_1 = - x_1y_1 - 2x_2 + 2y_2 - x_2y_2$
\item $N_2 = \frac12 - x_1 + y_2 - x_1y_2 - 2x_2 - 2y_1 + x_2y_1$
\item $N_3 = \frac12 - x_2 - y_1 + x_2y_1 - 2x_1 + 2y_2 - x_1y_2$
\item $N_4 = x_2y_2 - 2x_1 - 2y_1 + x_1y_1$
\item $N_5 = \frac12 + x_2 - y_1 - x_2y_1 - 2x_1 - 2y_2 + x_1y_2$
\item $N_6 = \frac12 - x_1 - y_2 + x_1y_2 + 2x_2 - 2y_1 - x_2y_1$
\item $N_7 = - x_1y_1 + 2x_2 - 2y_2 - x_2y_2$
\item $N_8 = \frac12 + x_1 - y_2 - x_1y_2 + 2x_2 + 2y_1 + x_2y_1$
\item $N_9 = \frac12 + x_2 + y_1 + x_2y_1 + 2x_1 - 2y_2 - x_1y_2$
\item $N_{10} = x_2y_2 + 2x_1 + 2y_1 + x_1y_1$
\item $N_{11} = \frac12 - x_2 + y_1 - x_2y_1 + 2x_1 + 2y_2 + x_1y_2$
\item $N_{12} = \frac12 + x_1 + y_2 + x_1y_2 - 2x_2 + 2y_1 - x_2y_1$
\end{itemize}

Now we want to consider two points moving towards a pair of opposite corners $p$ and $q$ with coordinates $p_1=p_2=-1/2$ and $q_1=-q_2=1/2$, along the diagonals of opposite faces; see Figure \ref{fig:cube-diagonals}. The diagonals are given by $x_1=x_2$ and $y_1=-y_2$. For convenience we introduce new coordinates $x_d= -x_1 = -x_2$ and $y_d=y_1=-y_2$ to describe pairs of points on the diagonals. Because we are going to consider points close to the corners $p$ and $q$ we will limit ourselves to the case $0<x_d,y_d<1/2$.

  \begin{figure}[hbt!]
  \centering
    \includegraphics{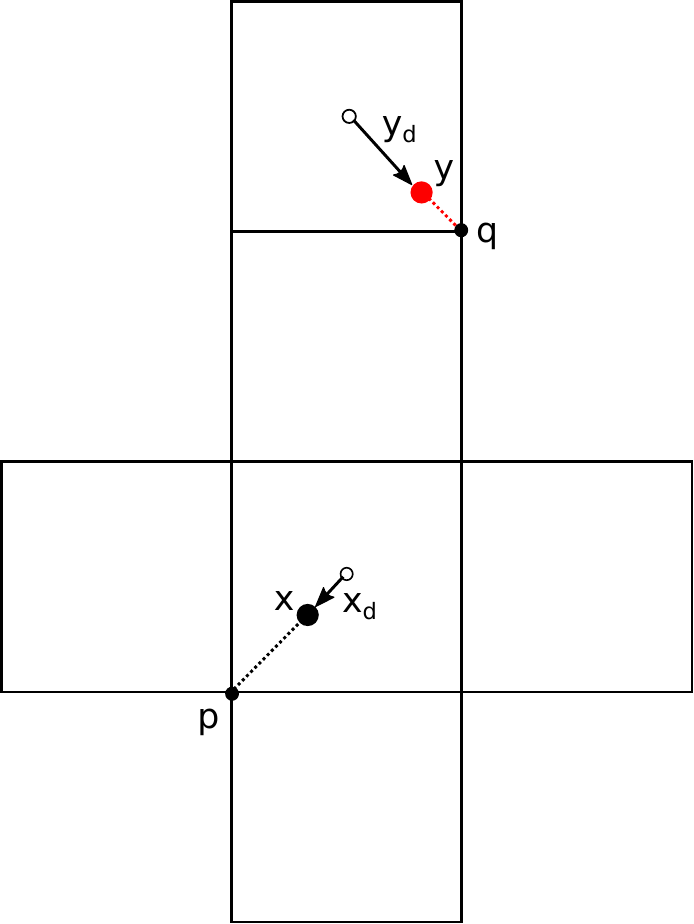}
    \caption{We introduce coordinates for points on opposite faces approaching the corner along the diagonal on the respective face.}
    \label{fig:cube-diagonals}
  \end{figure}

Assuming that $x_d= -x_1 = -x_2$ and $y_d=y_1=-y_2$, the normalized lengths above turn into the following expressions (substitute $x_1= -x_d$, $x_2= -x_d$, $y_1= y_d$ and $y_2= -y_d$).

\begin{itemize}
\item $N_1 = 2(x_d-y_d)$
\item $N_2 = \frac12 + 3(x_d-y_d) - 2x_dy_d$
\item $N_3 = \frac12 + 3(x_d-y_d) - 2x_dy_d$
\item $N_4 = 2(x_d-y_d)$
\item $N_5 = \frac12 + x_d + y_d + 2x_dy_d$
\item $N_6 = \frac12 - x_d - y_d + 2x_dy_d$
\item $N_7 = -2(x_d-y_d)$
\item $N_8 = \frac12 - 3(x_d-y_d) - 2x_dy_d$
\item $N_9 = \frac12 - 3(x_d-y_d) - 2x_dy_d$
\item $N_{10} = -2(x_d-y_d)$
\item $N_{11} = \frac12 - x_d - y_d + 2x_dy_d$
\item $N_{12} = \frac12 + x_d + y_d + 2x_dy_d$
\end{itemize}

If we further assume that the points are at the same distance from the corners (approaching the corners at the same rate) and set $z=x_d=y_d$ (see Figure \ref{fig:cube-sym-diagonals}) the expressions simplify further:

\begin{itemize}
\item $N_1 = 0$
\item $N_2 = \frac12 - 2z^2$
\item $N_3 = \frac12 - 2z^2$
\item $N_4 = 0$
\item $N_5 = \frac12 + 2z + 2z^2$
\item $N_6 = \frac12 - 2z + 2z^2$
\item $N_7 = 0$
\item $N_8 = \frac12 - 2z^2$
\item $N_9 = \frac12 - 2z^2$
\item $N_{10} = 0$
\item $N_{11} = \frac12 - 2z + 2z^2$
\item $N_{12} = \frac12 + 2z + 2z^2$
\end{itemize}

  \begin{figure}[hbt!]
  \centering
    \includegraphics{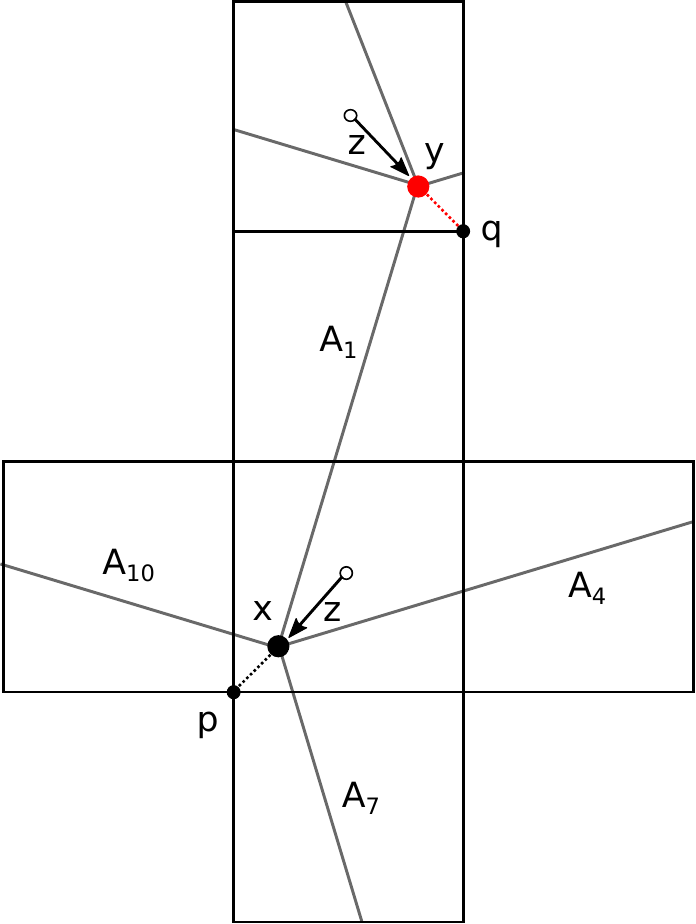}
    \caption{We introduce one coordinate $z$ for points on opposite faces approaching the corner along the diagonal on the respective face such that they are always at the same distance from the corners. There are precisely four geodesics for such pairs of points.}
    \label{fig:cube-sym-diagonals}
  \end{figure}

Note that $N_1$, $N_4$, $N_7$ and $N_{10}$ are equal and smaller than all other $N_i$ when $-x_1=-x_2=y_1=-y_2=z$, since all the other $N_i$ are positive, as long as $0<z<\frac12$.

Because the $N_i$ vary continuously, there is a small neighborhood around every pair $(x,y)$ with $0<-x_1=-x_2=y_1=-y_2=z<\frac12$ in which $N_1$, $N_4$, $N_7$ and $N_{10}$ are shorter than all other $N_i$. Let $U_A$ be the union of all such neighborhoods. Assuming that $(x,y)$ is in $U_A$, to determine all the geodesics between $x$ and $y$ we can immediately disregard all paths except $A_1$, $A_4$, $A_7$ and $A_{10}$.

Note that for pairs $(x,y)$ with $0<-x_1=-x_2=y_1=-y_2=z<\frac12$ all four paths $A_1$, $A_4$, $A_7$ and $A_{10}$ have the same length, which means that such pairs have exactly four preimages under $\pi \colon GW \to W \times W$. Therefore, we can construct a sequence $(s_A^i)$ converging to $(p,q)$ with coordinates
\[s_A^i=\left(\left(-\frac12+\frac1{5i},-\frac12+\frac1{5i}\right),\left(\frac12-\frac1{5i},-\frac12+\frac1{5i}\right)\right),\]
such that $\pi^{-1}\left(s_A^i\right)=\{A_1,A_4,A_7,A_{10}\}$.

Furthermore, for every sequence component $s_A^i$ we may construct two sequences $(r_{iI}^j)$ and $(r_{iII}^j)$ converging to $s_A^i$ with coordinates
\[r_{iI}^j=\left(\left(-\frac12+\frac1{5i}+\frac1{5j},-\frac12+\frac1{5i}+\frac1{5j}\right),\left(\frac12-\frac1{5i},-\frac12+\frac1{5i}\right)\right)\]
and
\[r_{iII}^j=\left(\left(-\frac12+\frac1{5i}-\frac1{5j},-\frac12+\frac1{5i}-\frac1{5j}\right),\left(\frac12-\frac1{5i},-\frac12+\frac1{5i}\right)\right),\]
such that $\pi^{-1}(r_{iI}^j)=\{A_1,A_4\}$ and $\pi^{-1}(r_{iII}^j)=\{A_7,A_{10}\}$. This is because $N_1=N_4<N_7=N_{10}$ if $x_d<y_d$ and $N_1=N_4>N_7=N_{10}$ if $x_d>y_d$.

Finally, for every sequence component $r_{iI}^j$ we can construct two sequences $(t_{ijI}^k)$ and $(t_{ijII}^k)$ converging to $r_{iI}^j$ such that $\pi^{-1}(t_{ijI}^k)=\{A_1\}$ and $\pi^{-1}(t_{ijII}^k)=\{A_4\}$, and similarly for $(r_{iII}^j)$. Intuitively, moving the first point of the pair $r_{iI}^j$ slightly in the direction of the path $A_1$ results in a pair with unique geodesic $A_1$ and similarly for $A_4$.

For the sake of concreteness, we may choose
\[t_{ijI}^k=\left(\left(-\frac12+\frac1{5i}+\frac1{5j},-\frac12+\frac1{5i}+\frac1{5j}+\frac1{100k}\right),\left(\frac12-\frac1{5i},-\frac12+\frac1{5i}\right)\right)\]
and
\[t_{ijII}^k=\left(\left(-\frac12+\frac1{5i}+\frac1{5j}+\frac1{100k},-\frac12+\frac1{5i}+\frac1{5j}\right),\left(\frac12-\frac1{5i},-\frac12+\frac1{5i}\right)\right).\]

To see why $\pi^{-1}(t_{ijI}^k)=\{A_1\}$ for $k$ big enough, it is enough to check that for $x_2=x_1+\epsilon$, $x_1=y_d-\delta$ and $y_1=-y_2=y_d$ we have $N_1<N_4,N_7,N_{10}$ (assuming that $\epsilon,\delta>0$ and small). This is a straightforward calculation, substituting the above expressions for $x$ and $y$ in the formulas for the appropriate $N_j$ and subtracting them pairwise. Similarly for $\pi^{-1}(t_{ijII}^k)=\{A_4\}$.

\vspace{.2cm}
\textbf{Step 2: Proof by contradiction that at least four $\mathbf{E_i}$ are necessary}
\vspace{.1cm}

Assume that we have a decomposition into disjoint locally compact sets $W\times W = \bigcup_{i=0}^2 E_i$ with a local section $s_i$ of $\pi \colon GW \to W \times W$ over each $E_i$. This will yield a contradiction to the continuity of the local sections $s_i$, analogously to the arguments in the proofs of the theorems \ref{thm:torus}, \ref{thm:klein} and \ref{thm:embedded-torus}.

First we show that all components of the sequences $(r_{iI}^j)$ and $(r_{iII}^j)$ are in the closure of two $E_i$. Assume that there exists a sequence component $r_{iI}^j$ which is in the interior of $E_0$, for instance. Then we may assume that $(t_{ijI}^k)$ and $(t_{ijII}^k)$ are in $E_0$ by taking a subsequence if necessary. By continuity, this would imply that $s_0(t_{ijI}^k)=A_1$ and $s_0(t_{ijII}^k)=A_4$ need to both converge to the same path $s_0(r_{iI}^j)$ as $k$ tends to infinity. This yields a contradiction, implying that every $r_{iI}^j$ lies in the closure of two $E_i$. The same argument applies to $(r_{iII}^j)$.

Next we show that all components of the sequence $(s_A^i)$ are in the closure of three $E_i$. Assume that there exists a sequence component $s_A^i$ which has a neighborhood $V_\epsilon$ which is fully contained in $E_0 \cup E_1$, for instance. We may assume that the sequences  $(r_{iI}^j)$ and $(r_{iII}^j)$ are contained in $V_\epsilon$. We showed that all components of the sequences $(r_{iI}^j)$ and $(r_{iII}^j)$ are in the closure of two $E_i$. Because we assumed that $V_\epsilon \subset E_0 \cup E_1$, this means that every $r_{iI}^j$ and $r_{iII}^j$ lies in the closure of $E_0$ (and $E_1$). By a diagonal argument we can construct two new sequences $(\tilde r_{iI}^j)$ and $(\tilde r_{iII}^j)$ converging to $s_A^i$ which are contained in $E_0$, and such that $\pi^{-1}(\tilde r_{iI}^j)\in\{A_1,A_4\}$ and $\pi^{-1}(\tilde r_{iII}^j)\in\{A_7,A_{10}\}$. This is because pairs sufficiently close to $r_{iI}^j$ and $r_{iII}^j$ cannot have other $A_i$ as geodesics, since the length of the $A_i$ varies continuously. By continuity, this would imply that $s_0(\tilde r_{iI}^j)\in\{A_1,A_4\}$ and $s_0(\tilde r_{iII}^j)\in\{A_7,A_{10}\}$ need to both converge to the same path $s_0(s_{A}^i)$ as $j$ tends to infinity. This yields a contradiction, implying that every $s_A^i$ lies in the closure of three $E_i$.

Finally, it just remains to show that this in turn implies that there is no way to choose a geodesic for $s_0((p,q))$ that would make $s_0$ continuous at the pair $(p,q)$ of opposite corners.

There are two other faces adjacent to the corner $p$, denoted $B$ and $C$ in Figure \ref{fig:cube-corners}. By using the 3-fold rotation symmetry of the cube around the corners $p$ and $q$ we can construct a neighborhood $U_B$ within which the geodesics are $B_1$, $B_4$, $B_7$ and $B_{10}$ and a neighborhood $U_C$ within which the geodesics are $C_1$, $C_4$, $C_7$ and $C_{10}$, as well as a sequence $(s_B^i)$ converging to $(p,q)$, such that $\pi^{-1}(s_B^i)=\{B_1,B_4,B_7,B_{10}\}$ and a sequence $(s_C^i)$ converging to $(p,q)$, such that $\pi^{-1}(s_C^i)=\{C_1,C_4,C_7,C_{10}\}$. Here the paths $B_i$ and $C_i$ are the result of rotating $A_i$ around the axis going through $p$ and $q$, and similarly for $U_B$, $U_C$, $(s_B^i)$ and $(s_C^i)$.

There are precisely six geodesics between $p$ and $q$ as seen in Figure \ref{fig:cube-corners}. We denote them $D_i$, $1\le i \le 6$, just as in the figure. The $D_i$ are limits of the paths $A_i$, $B_i$ and $C_i$ as pairs of points in $U_A$, $U_B$ and $U_C$ approach $(p,q)$. Specifically:

\begin{itemize}
\item The paths $A_1$, $A_4$, $A_7$ and $A_{10}$ converge to $D_3$, $D_4$, $D_6$ and $D_1$ respectively.
\item The paths $B_1$, $B_4$, $B_7$ and $B_{10}$ converge to $D_5$, $D_6$, $D_2$ and $D_3$ respectively.
\item The paths $C_1$, $C_4$, $C_7$ and $C_{10}$ converge to $D_1$, $D_2$, $D_4$ and $D_5$ respectively.
\end{itemize}

Just as we did above for $s_A^i$, we can show that all $s_B^i$ and all $s_C^i$ lie in the closure of $E_0$ (and of $E_1$ and $E_2$). Just as above, this allows us to use a diagonal argument to construct new sequences $(\tilde s_A^i)$, $(\tilde s_B^i)$ and $(\tilde s_C^i)$ contained in $E_0$ and converging to $(p,q)$ such that $\pi^{-1}(\tilde s_A^i)\in\{A_1,A_4,A_7,A_{10}\}$, $\pi^{-1}(\tilde s_B^i)\in\{B_1,B_4,B_7,B_{10}\}$ and $\pi^{-1}(\tilde s_C^i)\in\{C_1,C_4,C_7,C_{10}\}$.

By continuity, we need $s_0(\tilde s_A^i)$, $s_0(\tilde s_B^i)$ and $s_0(\tilde s_C^i)$ to converge to the same path $s_0((p,q))$. Thus, $s_0$ can only be continuous if the geodesic $s_0((p,q))$ lies in the following set $\{D_3, D_4 , D_6 , D_1\} \cap \{D_5, D_6 , D_2, D_3\} \cap \{D_1, D_2, D_4, D_5\} = \emptyset$. This yields a contradiction to the continuity of $s_0$.
\end{proof}  

%


  
    \begin{figure}[hbt!]
  \centering
    \includegraphics{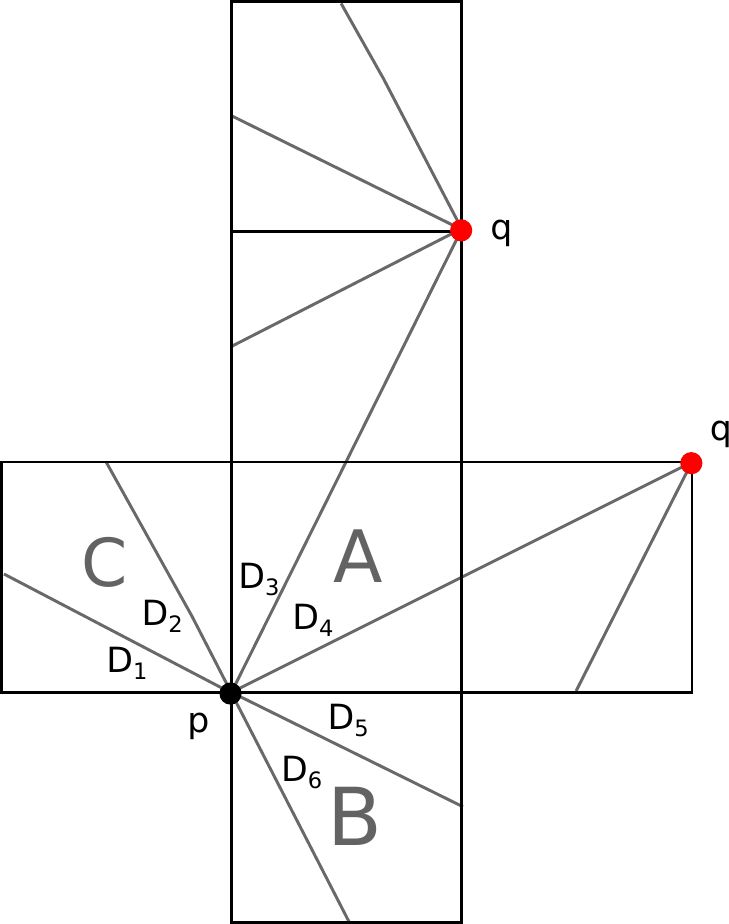}
    \caption{There are precisely six geodesics between opposite corners $p$ and $q$. We denote the faces adjacent to $p$ by $A$, $B$ and $C$.}
    \label{fig:cube-corners}
  \end{figure}

\begin{remark}
%
The map $GW \to W \times W$ seems to yield a level-wise stratified covering over a stratified space just as made explicit for the torus and the Klein bottle, but it appears to be a hard problem to explicitly understand the whole total cut locus. Nonetheless, it would be very interesting to understand the total cut locus and it seems likely that this would allow us to construct a geodesic motion planner with four sets, which would show $\GC(W)=3$.
\end{remark}

\begin{remark}
Presumably for the $n$-dimensional analogues $W^n$ of $W$ ($n$-spheres seen as boundary of $n+1$ cubes, with the flat metric) the geodesic complexity is $n+1$, while the topological complexity oscillates between 1 and 2. If this could be shown, it would yield another (maybe more natural) family of examples where the gap between TC and GC is unbounded.
\end{remark}

\section{Defining geodesic complexity using shortest paths}\label{sec:alternative-definition}

Recall that we say that a path $\gamma$ in a metric space $(X,\newd)$ is a geodesic if there exists a number $\lambda$ such that
\[ \newd(\gamma(t),\gamma(t')) = \lambda |t-t'|\]
for all $0 \le t < t' \le 1$.

It immediately follows from Definition \ref{def:length} that $\ell(\gamma)=\newd(\gamma(0),\gamma(1)) = \lambda$. It is also clear from the definition that $\newd(\gamma(0),\gamma(1))$ is the shortest length a path from $\gamma(0)$ to $\gamma(1)$ could possibly have.

\begin{definition}
We say that a path $\gamma$ in a metric space $(X,\newd)$ is a \textit{shortest path} if $\ell(\gamma)=\newd(\gamma(0),\gamma(1))$.
\end{definition}

\begin{definition}
Let $\gamma$ be a path in a metric space $(X,\newd)$. We say that $\gamma'$ is a \textit{reparametrization} of $\gamma$ if $\gamma'\circ\psi = \gamma$ for some non-decreasing surjective map $\psi\colon [0,1] \to [0,1]$.
\end{definition}

Geodesics are precisely the shortest paths which are parametrized proportional to arc length, i.e.\ have constant speed $\lambda$. Any non-trivial reparametrization of a geodesic is no longer a geodesic, but it will still be a shortest path, since the length of a path is independent of the parametrization. This rigidity is very useful in the lower bound arguments for level-wise stratified coverings in Section \ref{sec:cut-locus} and for the examples in sections \ref{sec:examples}, \ref{sec:embedded-torus} and \ref{sec:flat-spheres}. However, one might ask if allowing more flexibility in the parametrization might make it possible to find motion planners with lower complexity (meaning decompositions into fewer sets). In Theorem \ref{thm:reparametrization} below we show that this is not the case, but first we need some preliminary definitions.

By \cite[Chapter I, 1.22 Remark]{BH}, any shortest path can be reparametrized to be a geodesic as defined above: Let $\gamma$ be a path with $\ell(\gamma)=\newd(\gamma(0),\gamma(1))\ne0$. Then the map $\lambda\colon [0,1]\to[0,1]$ defined by
\[
\lambda(t) \coloneqq \frac{\ell(\gamma|_{[0,t]})}{\ell(\gamma)},
\]
yields a reparametrization $\gamma'$ which is a geodesic (concretely, we have $\gamma = \gamma'\circ\lambda$). If $\ell(\gamma) = 0$, meaning that $\gamma$ is constant, $\gamma$ is already a geodesic and we can set $\lambda = \text{id}|_{[0,1]}$. Note that in \cite{BH} geodesics are required to have unit speed and are defined on intervals of varying length. Because under our definition geodesic paths are always defined on the interval $[0,1]$, we need to divide by the total length in the formula for $\lambda(t)$ above. In fact, our definition of geodesic corresponds precisely to the definition of \textit{linearly reparametrized geodesic} in \cite{BH}, which are allowed to have any constant speed.

\begin{definition}
Let $(X,\newd)$ be a metric space. Denote by $\widetilde{GX}\subset PX$ the space of all shortest paths and denote by $\tilde\pi \colon \widetilde{GX} \to X \times X$ the restriction of the evaluation map.
\end{definition}

\begin{theorem}\label{thm:reparametrization}
Replacing $\pi \colon GX \to X \times X$ by $\tilde\pi \colon \widetilde{GX} \to X \times X$ in the definition of geodesic complexity would result in an equivalent definition. In other words, the locally compact sectional categories of $\pi$ and $\tilde\pi$ agree: $\enrsecat{\pi} = \enrsecat{\tilde\pi}$.
\end{theorem}

\begin{proof}

Because $GX \subset \widetilde{GX}$, it is immediate that $\enrsecat{\pi} \ge \enrsecat{\tilde\pi}$. It remains to show the reverse inequality.

We will show that, given a local section $\tilde s\colon E \to \widetilde{GX}$ of $\tilde\pi$, reparametrizing each shortest path $\tilde s(e)$ yields a local section $s\colon E \to GX$ of $\pi$. Concretely, we need to prove that letting $s(e)$ be the path such that $\tilde s(e) = s(e) \circ \lambda_e$ defines a continuous map $s\colon E \to GX$. Recall that $\lambda_e (t) = \frac{\ell(\tilde s(e)|_{[0,t]})}{\ell(\tilde s(e))}$ if $\ell(\tilde s(e)) \neq 0$ and $\lambda_e (t) = t$ otherwise.

To simplify the notation, we write $\tilde s(e) = \tilde s_e$ and $s(e) = s_e$.

Recall that the path space $PX$ is equipped by the compact-open topology and that, if $X$ is a metric space, the compact-open topology on $PX$ is induced by the supremum metric

\[
\dsup(\gamma_1,\gamma_2) = \sup_{t\in[0,1]} \newd( \gamma_1(t) , \gamma_2(t) )
.\]

The subspace topology on both $GX$ and $\widetilde{GX}$ is also induced by the supremum metric.

We need to show that, for any given $e_0$ in $E$, the map $s\colon E \to GX$ is continuous at $e_0$ with respect to the supremum metric on $GX$. It suffices to show that for all $\epsilon>0$ there exists a $\delta>0$ such that $\newd(s_{e_0}(t),s_e(t)) < \epsilon$ for all $t$, whenever $\newd(e_0,e) < \delta$.

Let $\epsilon>0$. Choose $\delta>0$ such that $\newd(\tilde s_{e_0}(t),\tilde s_e(t)) < \epsilon/5$ for all $t$, whenever $\newd(e_0,e) < \delta$. Such a $\delta$ exists because $\tilde s$ is continuous and by the definition of the supremum metric. 

Furthermore, we may assume that $\delta>0$ is small enough such that $\ell(\tilde s_e) \ne 0$ whenever $\newd(e_0,e) < \delta$: If $\ell(\tilde s_{e_0}) \ne 0$, such a $\delta$ can be found by the continuity of $\tilde s$. In the case when $\ell(\tilde s_{e_0}) = 0$, the inequality $\dsup(\tilde s_{e_0}, \tilde s_e) < \epsilon$ immediately implies $\dsup(s_{e_0},s_e) < \epsilon$, because $\tilde s_{e_0} = s_{e_0}$ is a constant path. Thus, in that case the continuity of $s$ at $e_0$ follows immediately from the continuity of $\tilde s$ at $e_0$.



Let $t$ be in $[0,1]$ and choose $t'$ and $t''$ such that $\lambda_{e_0}(t') = \lambda_{e}(t'') = t$. Then:
\begin{align*}
\newd(s_{e_0}(t), s_e(t))
= &
\newd( s_{e_0} (\lambda_{e_0} (t')), s_e (\lambda_{e} (t'')))\\
\le &
\newd( s_{e_0} (\lambda_{e_0} (t')),  s_e (\lambda_{e} (t')))
+
\newd( s_{e} (\lambda_{e} (t')),  s_e (\lambda_{e} (t''))) \\
= &
\underbrace{\newd( \tilde s_{e_0} (t'),  \tilde s_e (t'))}_{<\epsilon/5}
+
\newd( \tilde s_{e} (t'),  \tilde s_e (t''))
.
\end{align*}

Because $\tilde s_e$ is a shortest path, we can rewrite the second summand above as follows
\begin{align*}
\newd( \tilde s_{e} (t') &, \tilde s_e (t''))
=
| \newd( \tilde s_{e} (0), \tilde s_e (t')) - \newd( \tilde s_{e} (0), \tilde s_e (t'')) |\\
\le &
\underbrace{| \newd( \tilde s_{e} (0), \tilde s_e (t')) - \newd( \tilde s_{e_0} (0), \tilde s_{e_0} (t')) |}_{<2\epsilon/5 \;\; (\text{Inequality A})}
+
\underbrace{| \newd( \tilde s_{e_0} (0), \tilde s_{e_0} (t')) - \newd( \tilde s_{e} (0), \tilde s_e (t'')) |}_{<2\epsilon/5 \;\; (\text{Inequality B})}
.
\end{align*}

Inequality A follows from

\begin{align*}
\newd( \tilde s_{e} (0), \tilde s_e (t')) \le
 \underbrace{\newd( \tilde s_{e} (0), \tilde s_{e_0} (0))}_{<\epsilon/5} + \newd( \tilde s_{e_0} (0), \tilde s_{e_0} (t')) + \underbrace{\newd( \tilde s_{e_0} (t'), \tilde s_{e} (t'))}_{<\epsilon/5}
\end{align*}

and the analogous reverse inequality.

To show Inequality B note that
\[
\lambda_e(t) = \frac{\ell(\tilde s_e|_{[0,t]})}{\ell(\tilde s_e)} = \frac{\newd(\tilde s_e(0),\tilde s_e(t))}{\newd(\tilde s_e(0),\tilde s_e(1))}
\]
because $\tilde s_e$ is a shortest path.

Recall that $t = \lambda_{e_0}(t') = \lambda_e(t'')$ and thus
\[
\frac{\newd(\tilde s_{e_0}(0),\tilde s_{e_0}(t'))}{\newd(\tilde s_{e_0}(0),\tilde s_{e_0}(1))} = \frac{\newd(\tilde s_e(0),\tilde s_e(t''))}{\newd(\tilde s_e(0),\tilde s_e(1))}
.
\]

This implies that
\begin{align*}
\newd(\tilde s_{e_0}(0),\tilde s_{e_0}(t')) = &
 \newd(\tilde s_e(0),\tilde s_e(t'')) \frac{\newd(\tilde s_{e_0}(0),\tilde s_{e_0}(1))}{\newd(\tilde s_e(0),\tilde s_e(1))}\\
 = & \newd(\tilde s_e(0),\tilde s_e(t'')) \left( 1 +\frac{ \newd(\tilde s_{e_0}(0),\tilde s_{e_0}(1)) - \newd(\tilde s_e(0),\tilde s_e(1))}{\newd(\tilde s_e(0),\tilde s_e(1))} \right)\\
\le & \newd(\tilde s_e(0),\tilde s_e(t'')) + | \newd(\tilde s_{e_0}(0),\tilde s_{e_0}(1)) - \newd(\tilde s_e(0),\tilde s_e(1)) | \\
\le & \newd(\tilde s_e(0),\tilde s_e(t'')) + \underbrace{\newd(\tilde s_{e_0}(0),\tilde s_{e}(0))}_{<\epsilon/5} + \underbrace{\newd(\tilde s_{e_0}(1),\tilde s_e(1))}_{<\epsilon/5}
.
\end{align*}

Together with the analogous reverse inequality, this implies Inequality B.
\end{proof}

\section{Further work}\label{sec:discussion}

There are many open questions left to explore regarding geodesic complexity. For example, it is unclear to which extent classical bounds for topological complexity still hold (in some form) for geodesic complexity.

The most commonly used upper bound for topological complexity is $\TC(X)\le\text{dim}(X\times X)$ \cite{Far03}, where $\text{dim}(X\times X)$ is the covering dimension. We were not able to find metric spaces for which we can prove that $\GC(X)>\text{dim}(X\times X)$, which leaves open the possibility that the dimensional upper bound also holds for geodesic complexity. However, there does not seem to be a reason for $\GC(X)\le\text{dim}(X \times X)$ to hold, since the arguments used in the case of topological complexity do not apply in the geodesic setting.

\textbf{Question 1:} Does the bound $\GC(X) \le \text{dim}(X\times X)$ hold for sufficiently nice metric spaces $X$?

It should be noted that we do not mean homotopy dimension here, as is often done in the study of topological complexity. We know from Theorem \ref{thm:gap} that even contractible metric spaces can have arbitrarily high geodesic complexity.

Other commonly used bounds for topological complexity are given by the Lusternik-Schnirelmann category $\text{cat}(X)$, which is the smallest $k$ for which there is an open cover $X=\bigcup_{i=0}^k U_i$ such that each inclusion $U_i\hookrightarrow X$ is null-homotopic. The Lusternik-Schnirelmann category is a homotopy invariant which is closely related to topological complexity. In \cite{Far03} Farber shows that:

\[\text{cat}(X) \le \TC(X) \le \text{cat}(X\times X)\]

The lower bound trivially carries over to geodesic complexity since $\TC(X)\le\GC(X)$. However, the upper bound does not carry over: According to Theorem \ref{thm:gap}, there is a Riemannian metric $g_m$, such that $\GC(S^{n+1},g_m) \ge n$, while it is well-known that $\text{cat}(S^n\times S^n)=2$. Furthermore, Theorem \ref{thm:flat-sphere} also yields a counterexample, since the theorem states that $\GC(W)\ge3$, and yet $\text{cat}(W\times W)=\text{cat}(S^2\times S^2)=2$.

Possibly the bound can be recovered by replacing the Lusternik-Schnirelmann category with a geodesic version. Given a metric space $X$, let $\text{Gcat}(X)$ be the smallest $k$ for which there is an decomposition into locally compact sets $X=\bigcup_{i=0}^k E_i$ such that each $E_i\hookrightarrow X$ is null-homotopic along geodesics (meaning that the homotopy restricted to any point of $X$ yields a geodesic). It is easy to see that $\text{Gcat}(X)\le \GC(X)$ but the proof of $\TC(X) \le \text{cat}(X\times X)$ does not carry over to the geodesic case.

\textbf{Question 2:} Does the bound $\GC(X) \le \text{Gcat}(X\times X)$ hold for sufficiently nice metric spaces $X$?


\end{document}